\documentclass{amsart}

\usepackage{amsmath}
\usepackage{amsfonts}
\usepackage{amssymb}
\usepackage{amsthm}
\usepackage{comment}
\usepackage{epsfig}
\usepackage{psfrag}
\usepackage{mathrsfs}
\usepackage{amscd}
\usepackage[all]{xy}
\usepackage{rotating}
\usepackage{lscape}
\usepackage{amsbsy}
\usepackage{verbatim}
\usepackage{moreverb}
\usepackage{color}
\usepackage{bbm}
\usepackage{eucal}
\usepackage{stmaryrd}

\usepackage{tikz-cd}
\usetikzlibrary{patterns,shapes.geometric,arrows,decorations.markings}
\usepackage{tikz-3dplot}

\usepackage{caption}
\usepackage{subcaption}

\colorlet{lightgray}{black!15}

\tikzset{->-/.style={decoration={
  markings,
  mark=at position .5 with {\arrow{>}}},postaction={decorate}}}
\tikzset{midarrow/.style={decoration={
    markings,
    mark=at position {#1} with {\arrow{>}}},postaction={decorate}}}

\newtheorem{theorem}{Theorem}[section]
\newtheorem{prop}[theorem]{Proposition}
\newtheorem{lemma}[theorem]{Lemma}
\newtheorem{cor}[theorem]{Corollary}

\theoremstyle{definition}
\newtheorem{definition}[theorem]{Definition}

\newtheorem{observation}[theorem]{Observation}

\newtheorem{terminology}[theorem]{Terminology}
\newtheorem{remark}[theorem]{Remark}
\newtheorem{example}[theorem]{Example}

\newtheorem{notation}[theorem]{Notation}
\newtheorem{l.notation}[theorem]{Local Notation}

\newtheorem{convention}[theorem]{Convention}

\theoremstyle{remark}

\definecolor{orange}{rgb}{.95,0.5,0}
\definecolor{light-gray}{gray}{0.75}
\definecolor{brown}{cmyk}{0, 0.8, 1, 0.6}
\definecolor{plum}{rgb}{.5,0,1}

\DeclareMathOperator{\pr}{\mathsf{pr}}
\DeclareMathOperator{\ev}{\mathsf{ev}}

\DeclareMathOperator{\Alg}{\sf Alg}

\DeclareMathOperator{\Mod}{\sf Mod}

\DeclareMathOperator{\CAlg}{\sf CAlg}

\DeclareMathOperator{\Aut}{\sf Aut}
\DeclareMathOperator{\colim}{{\sf colim}}

\DeclareMathOperator{\Hom}{\sf Hom}
\DeclareMathOperator{\End}{\sf End}

\DeclareMathOperator{\Fun}{{\sf Fun}}

\DeclareMathOperator{\Map}{{\sf Map}}

\DeclareMathOperator{\refl}{\sf refl}

\DeclareMathOperator{\Cat}{{\sf Cat}}

\DeclareMathOperator{\CAT}{{\sf CAT}}

\DeclareMathOperator{\op}{\mathsf{op}}

\DeclareMathOperator{\conf}{\mathsf{Conf}}

\DeclareMathOperator{\Spaces}{\mathsf{Spaces}}

\DeclareMathOperator{\fr}{\sf fr}

\DeclareMathOperator{\Bord}{{\sf Bord}_1^{\fr}}

\DeclareMathOperator{\BO}{\sf BO}
\DeclareMathOperator{\RP}{\RR\PP}

\def\ot{\otimes}

\DeclareMathOperator{\oo}{\infty}

\newcommand{\lag}{\langle}
\newcommand{\rag}{\rangle}

\newcommand{\w}{\widetilde}
\newcommand{\un}{\underline}
\newcommand{\ov}{\overline}

\newcommand{\xra}{\xrightarrow}
\newcommand{\xla}{\xleftarrow}

\def\cA{\mathcal A}\def\cB{\mathcal B}\def\cC{\mathcal C}
\def\cE{\mathcal E}\def\cH{\mathcal H}
\def\cK{\mathcal K}
\def\cM{\mathcal M}
\def\cR{\mathcal R}\def\cS{\mathcal S}
\def\cU{\mathcal U}\def\cV{\mathcal V}\def\cW{\mathcal W}\def\cX{\mathcal X}

\def\DD{\mathbb D}

\def\LL{\mathbb L}
\def\PP{\mathbb P}
\def\RR{\mathbb R}\def\SS{\mathbb S}

\def\ZZ{\mathbb Z}

\def\sB{\mathsf B}

\def\sO{\mathsf O}
\def\sR{\mathsf R}\def\sT{\mathsf T}
\def\sU{\mathsf U}

\def\bDelta{\mathbf\Delta}

\DeclareMathOperator{\Obj}{\mathsf{Obj}}

\DeclareMathOperator{\PShv}{\mathsf{PShv}}
\DeclareMathOperator{\uno}{\mathbbm{1}}

\DeclareMathOperator{\rigid}{\sf rigid}
\DeclareMathOperator{\id}{\sf id}

\DeclareMathOperator{\Mor}{\sf Mor}

\def\bfX{\boldsymbol{\mathfrak X}}

\newcommand{\bit}[1]{\textbf{\textit{#1}}}

\DeclareMathOperator{\lacts}{\curvearrowright}

\DeclareMathOperator{\Ass}{\sf Assoc}

\begin{document}

\title{Symmetries of a rigid braided category}

\author{David Ayala \& John Francis}

\address{Department of Mathematics\\Montana State University\\Bozeman, MT 59717}
\email{david.ayala@montana.edu}
\address{Department of Mathematics\\Northwestern University\\Evanston, IL 60208}
\email{jnkf@northwestern.edu}
\thanks{DA was supported by the National Science Foundation under awards 1812055 and 1945639. JF was supported by the National Science Foundation under award 1812057. This material is based upon work supported by the National Science Foundation under Grant No. DMS-1440140, while the authors were in residence at the Mathematical Sciences Research Institute in Berkeley, California, during the Spring 2020 semester.}

\begin{abstract}
We identify natural symmetries of each rigid higher braided category.
Specifically, we construct a functorial action by the continuous group $\Omega \sO(n)$ on each $\cE_{n-1}$-monoidal $(g,d)$-category $\cR$ in which each object is dualizable (for $n\geq 2$, $d \geq 0$, $d \leq g \leq \infty$).  
This action determines a canonical action by the continuous group $\Omega \RR\PP^{n-1}$ on the moduli space of objects of each such $\cR$.
In cases where the parameters $n$, $d$, and $g$ are small, 
we compare these continuous symmetries to known symmetries, which manifest as categorical identities.  

\end{abstract}



\maketitle

\tableofcontents

\section*{Introduction}

This paper identifies natural symmetries of each \bit{rigid higher braided-monoidal category}.
Here, by {\it higher category} we mean \bit{$(g,d)$-category}\footnote{
Recall that a \bit{$(g,d)$-category} is a (weak) $g$-category in which, for each $d<k\leq g$, each $k$-morphism is invertible.
So, a $(1,1)$-category is an ordinary category.
The collection of $(g,d)$-categories organizes as a $(g+1,1)$-category $\Cat_{(g,d)}$.
}
for specified $d\geq 1$ and $d\leq g \leq \infty$;
by {\it higher braided-monoidal category} we mean \bit{$(n-1)$-monoidal higher category}\footnote{
Recall the operad $\cE_{n-1}$ of little $(n-1)$-disks.
An \bit{$(n-1)$-monoidal $(g,d)$-category} is a $\cE_{n-1}$-algebra in $\Cat_{(g,d)}$.
} 
for a specified $n \geq 2$.
By \bit{rigid} we mean that each object has a monoidal dual (see Definition~\ref{d4} for a precise definition).
\begin{example}
Here are examples of higher braided-monoidal categories, all for $d=1$.
\begin{enumerate}

\item
For $n-2>g$, an $(n-1)$-monoidal $g$-category is a symmetric monoidal $(g,1)$-category. For $g=1$ and $n>3$, this is simply a symmetric monoidal category, in the ordinary sense. 
\\
In particular, for $\Bbbk$ a commutative ring, an example is the symmetric monoidal category $\Mod_{\Bbbk}$ in which an object is a $\Bbbk$-module, a morphism is a $\Bbbk$-linear map, and the symmetric monoidal structure is given by $(M,N)\mapsto M \underset{\Bbbk} \ot N$, tensoring over $\Bbbk$.

\item
For $n=2$, a $1$-monoidal $g$-category is a monoidal $(g,1)$-category. 
For $n=2$ and $g=1$, this is a monoidal category in the ordinary sense. 
\\
In particular, for $\Bbbk$ a commutative ring and $A$ a $\Bbbk$-algebra, 
an example of a 1-monoidal $(1,1)$-category is ${\sf BiMod}_A$, in which an object is an $(A,A)$-bimodule, a morphism is a $(A,A)$-linear map, and the monoidal structure is given by $(M,N)\mapsto M \underset{A}\ot N$, tensor product over $A$.

\item
Of particular interest is $n=3$ and $g=1$, in which case an $2$-monoidal $1$-category is a braided-monoidal category in the ordinary sense. 
For example, for $\cH$ a quasi-triangular Hopf algebra over $\Bbbk$ with universal $R$-element $R\in \cH\ot \cH$ (such as $U_qG$), the braided-monoidal category ${\sf Rep}(\cH)$ has as objects the $\Bbbk$-linear representations of $\cH$; morphisms are $\cH$-linear maps; the monoidal structure is given by $(M,N)\mapsto M\ot_{\Bbbk} N$ tensor product over $\Bbbk$; and the braiding is the composite of multiplication by $R$ and permuting the tensor factors:
\[
\beta_{M,N} \colon M\underset{\Bbbk}\ot N\xra{R} M\underset{\Bbbk}\ot N \xra{\cong} N \underset{\Bbbk}\ot M~.
\]

\item 
Consider the $(\infty,1)$-category $\cM{\sf od}_{\Bbbk}$ of chain complexes over a commutative ring $\Bbbk$ localized at quasi-isomorphisms.
Derived tensor product $(U,V)\mapsto U \underset{\Bbbk}{\overset{\LL}\ot} V$ over $\Bbbk$ defines a symmetric monoidal structure on $\cM {\sf od}_{\Bbbk}$.
\begin{enumerate}
\item
Let $A$ be an $\cE_n$-algebra in $\cM{\sf od}_{\Bbbk}$. If the characteristic of $\Bbbk$ is zero and $n\geq 2$, by formality (\cite{kontsevich}, \cite{tamarkin}) this is equivalent to an $n$-Poisson algebra: a compatible pair of a commutative algebra structure and an $(n-1)$-shifted Lie algebra structure on $A$.
An example of an $(n-1)$-monoidal $(\infty,1)$-category is
${\sf L}\cM{\sf od}_{A}$:
an object is a left $A$-module, 
a morphism is an $A$-linear map, 
the $(n-1)$-monoidal structure is given by $(M,N)\mapsto M \underset{A}\ot N$, tensor product over $A$.

\item 
Let $A$ be an $\cE_{n-1}$-algebra in $\cM{\sf od}_{\Bbbk}$. 
An example of an $(n-1)$-monoidal $(\infty,1)$-category is $\cM{\sf od}^{\cE_{n-1}}_A$:
an object is an $\cE_{n-1}$-$A$-module,
a morphism is a $\cE_{n-1}$-$A$-linear map,
the $(n-1)$-monoidal structure is given by $(M,N)\mapsto M \underset{A} \ot N$, tensor product over $A$.

\end{enumerate}

\item
Generally, by Dunn's additivity (\cite{dunn}, and Theorem~5.1.2.2 ~\cite{HA}), for $n>2$ an $(n-1)$-monoidal $(g,1)$-category is a monoid-object among $(n-2)$-monoidal $(g,1)$-categories.  
In other words, an $(n-1)$-monoidal $(g,1)$-category is a $(g,1)$-category together with $(n-1)$ compatible monoidal structures on it.

\end{enumerate}

\end{example}

We state our main result informally here (see Theorem~\ref{t24} for a precise statement in the case that $d=1$, and Theorem~\ref{t35} for the general case).
For $X$ a space and $1\leq g\leq \infty$, denote by $\pi_{\leq g}X$ the \bit{fundamental $g$-groupoid of $X$}.\footnote{This is the initial \bit{$g$-type} under $X$ (i.e., it is the initial space under $X$ whose homotopy groups in degrees above $g$ vanish).}

\begin{theorem}
\label{t13}
Let $n\geq 2$ and $d \geq 1$ and $d\leq g \leq \infty$.
Let $\cR$ be a rigid $(n-1)$-monoidal $(g,d)$-category.
\begin{enumerate}
\item
The continuous group $\pi_{\leq g}\Omega \sO(n)$ acts on the $(n-1)$-monoidal $(g,d)$-category $\cR$.

\item
This action canonically descends to actions by the continuous group $\pi_{\leq g}\Omega\RP^{n-1}$ on its moduli space $\Obj(\cR)$ of objects as well as on its $(g,d-1)$-category $\Mor(\cR)$ of $1$-morphisms.  

\item
With respect to these actions, the source and target maps
\[
\Obj(\cR )\xla{~\ev_s~} \Mor(\cR)\xra{~\ev_t~}\Obj(\cR)
\]
are canonically $\pi_{\leq g}\Omega \SS^{n-1}$-equivariant.
\\
In particular, for each pair of objects $U,V\in \Obj(\cR)$, there is a canonical action by the continuous group $\pi_{\leq g-1}\Omega^2 \SS^{n-1}$ on the $(g-1,d-1)$-category $\un{\Hom}_{\cR}(U,V)$.  

\item
All of these actions are functorial among morphisms in the $(n-1)$-monoidal $(g,d)$-category $\cR$.

\end{enumerate}

\end{theorem}

After the following result, which is assembled in~\S\ref{sec.final}, we hereafter only discuss the case of Theorem~\ref{t13} in which $d=1$ and $g=\infty$.
\begin{prop}
\label{t37}
Theorem~\ref{t13} follows from its cases in which $d=1$ and $g = \infty$.  

\end{prop}

As detailed in~\S1, 
rotating $(n-1)$-monoidal structures and taking 1-categorical opposites define an action by the product of orthogonal groups $\sO(n-1)\times \sO(1)$ on the $\infty$-category of $(n-1)$-monoidal $(\infty,1)$-categories with duals:
\begin{equation}
\label{e32}
\sO(n-1)\times \sO(1)
~\lacts~
\Alg^{\rigid}_{{n-1}}(\Cat_{(\infty,1)})
~.
\end{equation}
In~\S\ref{sec.main.proof}, we show that Theorem~\ref{t13} (for $d=1$ and $g=\infty$) quickly follows from the following result (restated, and proved, as Theorem~\ref{t19}).

\begin{prop}
\label{t10}
There is a canonical $\sO(n-1)\times \sO(1)$-equivariant extension of the action map of~(\ref{e32}):
\[
\xymatrix{
&
\sO(n-1)\times \sO(1)
\times 
\Alg_{{n-1}}^{\rigid}(\Cat_{(\infty,1)})
\ar[dr]^-{\rm (\ref{e32})'s~action}
\ar[dl]_-{\oplus \times \id }
&
\\
\sO(n)
\times 
\Alg_{{n-1}}^{\rigid}(\Cat_{(\infty,1)})
\ar@{-->}[rr]^-{\sO(n-1)\times \sO(1)- \rm equivariant}
&
&
\Alg_{{n-1}}^{\rigid}(\Cat_{(\infty,1)})
~.
}
\]

\end{prop}

The $n=2$ case of Proposition~\ref{t10} is implied by the following result of~\cite{n=2}.
\begin{theorem}[\cite{n=2}]
\label{t11}
The action $\sO(1)\times \sO(1) \underset{(\ref{e11})}\lacts \Alg^{\rigid}(\Cat_{(\infty,1)})$ extends as an action $\sO(2)\lacts \Alg^{\rigid}(\Cat_{(\infty,1)})$.  
\end{theorem}

\begin{remark}
\label{r1}
For $n>2$, we are confident that the statement of Proposition~\ref{t10} can be strengthened as an extension of the $\sO(n-1)\times \sO(1)$-action on $\Alg_{{n-1}}^{\rigid}(\Cat_{(\infty,1)})$ to an $\sO(n)$-action on $\Alg_{{n-1}}^{\rigid}(\Cat_{(\infty,1)})$.
For our purposes, we are content with the functorial $\Omega \sO(n)$-action on each $(n-1)$-monoidal $(\infty,1)$-category with duals (as in Theorem~\ref{t13}).
Nevertheless, it would be interesting to extend the methods presented in this article to attain such an $\sO(n)$-action.

\end{remark}

\begin{remark}
Let $n\geq 2$.
Consider the tangle higher-braided category $\Bord(\RR^{n-1})$.
This is an $(n-1)$-monoidal $(\infty,1)$-category.
Its space of objects is the moduli space of proper 0-submanifolds $W^0\subset \RR^{n-1}$ equipped with a {\it framing}, which is a null-homotopy of the constant map $W \xra{!} \ast \xra{\lag \RR^1 \subset \RR^n \rag} \RR\PP^{n-1}$.
Its space of morphisms is the moduli space of proper 1-submanifolds with boundary $W^1 \subset \RR^{n-1}\times \DD^1$ equipped with a {\it framing}, which is a null-homotopy of the Gauss map $W \xra{\tau_W} \RR\PP^{n-1}$.
This $(n-1)$-monoidal $(\infty,1)$-category $\Bord(\RR^{n-1})$ is rigid. 
Now, the space of null-homotopies of a map $W\to \RR\PP^{n-1}$ is a torsor for the continuous group $\Map(W,\Omega \RR\PP^{n-1})$.
The actions by $\Omega \RR\PP^{n-1}$ on the spaces $\Obj\bigl( \Bord(\RR^{n-1}) \bigr)$ and $\Mor\bigl( \Bord(\RR^{n-1}) \bigr)$ of Theorem~\ref{t13}(2) are given by rotation of framing, namely, as the canonical action of $\Omega \RR\PP^{n-1}$ on spaces of null-homotopies of maps to $\RR\PP^{n-1}$.  

\end{remark}

\begin{remark}
The cobordism hypothesis (see~\cite{baezdolan},~\cite{cobordism}) asserts an action $\sO(n) \lacts \Obj(\bfX)$ on the moduli space of objects in a symmetric monoidal $(\infty,n)$-category $\bfX$ in which each object is dualizable and each $k$-morphism has a left and right adjoint (for $0<k<n$).  
Entertain the example $\bfX = {\sf Morita}_{n}^{(n-1) \sf red}$, which is a symmetric monoidal Morita $(\infty,n)$-category of rigid $(n-1)$-monoidal $(\infty,1)$-categories.
Equivalences in ${\sf Morita}_{n}^{(n-1) \sf red}$ are Morita-equivalences between rigid $(n-1)$-monoidal $(\infty,1)$-categories.  
In this instance, the $\sO(n)$-action asserted by the cobordism hypothesis determines, for each rigid $(n-1)$-monoidal $(\infty,1)$-category $\cR$, an action $\Omega \sO(n) \to \Aut_{{\sf Morita}_{n}^{(n-1) \sf red}}(\cR)$, functorially in Morita-equivalences between such $\cR$.
Theorem~\ref{t13} enhances this action in two ways.
\begin{enumerate}
\item
It lifts this $\Omega \sO(n)$-action from being by Morita-equivalences, to being by $(n-1)$-monoidal $(\infty,1)$-equivalences:
\[
\Omega \sO(n) \to 
\Aut_{\Alg_{n-1}(\Cat_{(\infty,1)})}(\cR)
\longrightarrow
\Aut_{{\sf Morita}_{n}^{(n-1) \rm reduced}}(\cR)
~.
\]

\item
It extends the functoriality of this $\Omega \sO(n)$-action from Morita-equivalences to $(n-1)$-monoidal $(\infty,1)$-functors.  

\end{enumerate}

\end{remark}

\subsection{The action $\pi_1 \RR\PP^{n-1} \lacts \pi_0 \Obj(\cR)$}
\label{sec.involution}

Let $\cR$ be a rigid $(n-1)$-monoidal $(\infty,1)$-category.
Here we explain an action $\ZZ\lacts \Obj(\cR)$ that generates the action $\pi_1 \RR\PP^{n-1} = \pi_0 \Omega \RR\PP^{n-1} \lacts \pi_0 \Obj(\cR)$.

Each $v\in \SS^{n-2}$ determines an underlying monoidal category $(\cR,\overset{v}\ot)$ (see~(\ref{e25})), which is rigid.
By construction the monoidal $(\infty,1)$-categories $(\cR, \overset{v}\ot)$ and $(\cR , \overset{-v} \ot)$ are monoidally opposite to one another, continuously in $v\in \SS^{n-2}$:
\begin{equation}
\label{e43}
(\cR, \overset{v}\ot)^{{\sf rev}}
~\simeq~
(\cR, \overset{-v}\ot)
~,\qquad
(~
{\rm continuously~in~}
v\in \SS^{n-2}
~)
~.
\end{equation}
Taking right duals in $(\cR,\overset{v}\ot)$ defines an equivalence between the monoidal category $(\cR,\overset{v}\ot)$ and the categorical opposite of the monoidal category $(\cR , \overset{-v}\ot)$, continuously in $v\in \SS^{n-2}$:
\[
\rho_v\colon
(\cR,\overset{v}\ot)
\xra{~\simeq~}
(\cR^{\op_1} , \overset{-v}\ot)
~,\qquad
R
\mapsto
R^{\sR_v}
\qquad
(~
{\rm continuously~in~}
v\in \SS^{n-2}
~)
~.
\]
This equivalence restricts as an equivalence between their groupoids of objects, which are identical with one another:
\[
\Obj(\rho_v) \colon
\Obj(\cR)
\xra{~\simeq~}
\Obj(\cR^{\op_1})
\simeq
\Obj(\cR)
\qquad
(~
{\rm continuously~in~}
v\in \SS^{n-2}
~)
~.
\]
As $v\in \SS^{n-2}$ varies, this defines a map between spaces
\[
\SS^{n-2}
\longrightarrow
\Aut_{\Spaces}\bigl( \Obj(\cR) \bigr) 
~,\qquad
v
\longmapsto
\Obj(\rho_v)
~.
\]
So each $v\in \SS^{n-2}$ determines an action
\begin{equation}
\label{e44}
\ZZ
~\underset{\Obj(\rho_v)} \lacts~
\Obj(\cR)
~.
\end{equation}
Note that the identity~(\ref{e43}) lends the identity
\begin{equation}
\label{e45}
(\rho_v)^{-1} 
~\simeq~ 
\rho_{-v}
\qquad
(~
{\rm continuously~in~}
v\in \SS^{n-2}
~)
~.
\end{equation}
Because ${\SS^{n-2}}_{/\pm} \simeq \RR\PP^{n-2}$ is path-connected, the action~(\ref{e44}) is independent of $v\in \SS^{n-2}$ up to pre-composing with an automorphism of $\ZZ$.
Furthermore, provided $n>2$, there is a path $[0,1]\xra{\gamma} \SS^{n-2}$ from $\vec{1}\in \SS^{n-2}$ to $-\vec{1}\in \SS^{n-2}$.
So, for $n>2$, 
\[
[0,1] \longrightarrow \Aut_{\Spaces} \bigl( \Obj(\cR) \bigr)
~,\qquad
t
\mapsto
\Obj(\rho_{\gamma(t)})
\circ
\Obj(\rho_{\vec{1}})
~,
\]
is a path of automorphisms between $\Obj(\rho_{\vec{1}})^{\circ 2}$ and $\id_{\Obj(\cR)}$.
Consequently, provided $n>2$, the element $\pi_0 \Obj(\rho_{\vec{1}})^{\circ 2} \in \Aut_{\sf Sets} \bigl(\pi_0 \Obj(\cR) \bigr)$ has order 2.

In summary, the equivalence $\Obj(\cR) \xra{\Obj(\rho_{\vec{1}})}  \Obj(\cR)$ 
defines a canonical action $\ZZ \lacts  \Obj(\cR)$;
provided $n>2$, the resulting action $\ZZ\lacts \pi_0 \Obj(\cR)$ descends along the surjection $\ZZ \to \sO(1)$ as an action $\sO(1) \lacts \pi_0 \Obj(\cR)$.

\subsection{Classical case of ordinary braided-monoidal categories}

Here, we make Theorem~\ref{t13} explicit in the context of ordinary braided-monoidal categories.
Let $\cR$ be an ordinary braided-monoidal category (i.e., $\cR$ is a $\cE_2$-monoidal $(1,1)$-category) that is rigid.
The actions of Theorem~\ref{t13} can be summarized as
\begin{enumerate}

\item
an action by the group $\sB \ZZ \rtimes \sO(1) 
=
\pi_{\leq 1} \Omega \RR\PP^{2}$ on the 1-groupoid $\Obj(\cR)$;

\item
a compatible action by the group $\sB \ZZ = \pi_{\leq 1} \Omega \SS^{2}$ on the braided-monoidal category $\cR$;

\item
a compatible action by the group $\ZZ_{/2\ZZ} = \pi_{\leq 1} \Omega \sO(3)$ on the braided-monoidal category $\cR$.

\end{enumerate}
We explain these actions presently.
\begin{enumerate}

\item
The involution $\sO(1) \lacts  \Obj(\cR)$ can be described as follows.
Recall from~\S\ref{sec.involution} the action $\ZZ\lacts \Obj(\cR)$ generated by $\Obj(\rho_{\vec{1}})$, taking right duals in the underlying monoidal category $(\cR , \overset{\vec{1}}\ot)$.
There, it is also explained how the path $[0,1] \xra{t\mapsto e^{\pi i t}} \SS^1$ witnesses an identification $\Obj(\rho_{\vec{1}})^{\circ 2} \simeq \id$.
Using that $\Obj(\cR)$ is a 1-groupoid, this identification canonically extends as a homotopy coherent factorization of $\ZZ \xra{\bigl\lag \Obj( \rho_{\vec{1}}) \bigr\rag } \Aut_{\Spaces_{\leq 1}}\bigl( \Obj(\cR) \bigr)$ through 
$\sO(1) \to \Aut_{\Spaces_{\leq 1}}\bigl( \Obj(\cR) \bigr)$.

\item
We explain here how taking double right duals {\it in varying monoidal directions} defines an action $\ZZ\lacts V$ on each object $V\in \cR$, functorially in $V\in \cR$.
\begin{itemize}
\item[]
Namely, we seek a section to the forgetful functor $\Mod_{\ZZ}(\cR) \xra{\rm forget} \cR$.  
Through the equivalence $\Mod_{\ZZ}(\cR)
:=
\Fun_{\Cat_{(1,1)}}(\sB \ZZ , \cR)
\simeq
\Fun_{\Cat_{(1,1)}}( \SS^1 , \cR)$, 
the sought section is adjoint to the functor
\[
\SS^1
\longrightarrow
\Aut_{\Cat_{(1,1)}}(\cR)
~,\qquad
v
\mapsto
\rho_{-v} \circ \rho_{\vec{1}}
~,
\]
which is indeed a section because, via the identity~(\ref{e45}), $\rho_{-\vec{1}} \circ \rho_{\vec{1}} = \id$.

By construction, the resulting action $\sB \ZZ \lacts \Obj(\cR)$ intertwines with the $\sO(1)$-action just above as an action $\sB \ZZ \rtimes \sO(1) \lacts \Obj(\cR)$.

\end{itemize}

\item
We explain here how rotating the braiding defines an involution $\ZZ_{/2\ZZ} \lacts \cR$ on the braided-monoidal category $\cR$, which fixes the underlying $(1,1)$-category.
\begin{itemize}

\item[]
Taking orbits for the action $\sB \ZZ \simeq {\sf SO}(2) \lacts \cE_2$ defines a section of the forgetful functor
\[
\Fun\bigl( {\sf SO}(2) ,  \Alg_{2}(\Cat_{(1,1)}) \bigr)
~\simeq~
\Mod_{\ZZ}\bigl( \Alg_{2}(\Cat_{(1,1)}) \bigr)
\xra{~\rm forget~}
\Alg_{2}(\Cat_{(1,1)})
~.
\]
The value of this section on $\cR$ is the automorphism $\cR \xra{\alpha} \cR$ of the braided-monoidal category $\cR$ described as follows.
This automorphism $\alpha$ restricts as the identity on the underlying category: 
$\alpha(1) = \id_{\cR}$.
This automorphism $\alpha$ restricts as the automorphism $\alpha(2)$ of the braided-monoidal structure $\underset{\cR}\ot \in \Hom_{\Cat_{(1,1)}}(\SS^1 \underset{\Sigma_2} \times \cR^{\times 2} , \cR)$,
\[
\alpha(2)
~\in~ 
\Aut_{\Hom_{\Cat_{(1,1)}}(\SS^1 \underset{\Sigma_2} \times \cR^{\times 2}  , \cR)}(\underset{\cR}\ot)
~\simeq~
\Omega_{\underset{\cR}\ot} 
\Hom_{\Cat_{(1,1)}}( \SS^1 \underset{\Sigma_2} \times \cR^{\times 2}  , \cR)
~,
\]
that is adjoint to the composite $\Sigma_2$-invariant functor
\[
\alpha(2)
\colon
\SS^1
\times
\Bigl(
\SS^1 
\times \cR^{\times 2} 
\Bigr)
~=~
\Bigl(
\SS^1
\times
\SS^1 
\Bigr)
\times \cR^{\times 2} 
\xra{
{\rm multiply} \times \id
}
\SS^1 
\times \cR^{\times 2} 
\xra{~\underset{\cR}\ot~}
\cR
~.
\]
Informally, for each $U,V\in \cR$, the automorphism $\alpha$ carries the braiding $U\ot V \xra{\beta^V_U} V \ot U$ to $U \ot V \xra{\beta^V_U \circ ( \beta^U_V \circ \beta^V_U)} V \ot U$.

Now, by construction, the automorphism $\rho_{\vec{1}}^{\circ 2}$ of the braided-monoidal category $\cR$ conjugates the automorphism $\alpha$ to $\alpha^{-1}$:
there is an invertible 2-cell witnessing the diagram among braided-monoidal categories
\begin{equation}
\label{e50}
\xymatrix{
\cR
\ar[rr]^-{\alpha}
\ar[d]_-{\rho_{\vec{1}}^{\circ 2}}
&&
\cR
\ar[d]^-{\rho_{\vec{1}}^{\circ 2}}
\\
\cR
&&
\cR
\ar[ll]_-{\alpha}
}
\end{equation}
as commutative.\footnote{This diagram presents an extension of the pointed map $\SS^1 \vee \SS^1 \xra{\lag \alpha , \rho_{\vec{1}}^{\circ 2} \rag} {\sf BAut}_{\Alg_2(\Cat_{(1,1)})}(\cR)$ to a pointed map from the Klein bottle ${\sf Klein} \to {\sf BAut}_{\Alg_2(\Cat_{(1,1)})}(\cR)$.
For $n=3$, $d=1$ and general $1\leq g \leq \infty$, Theorem~\ref{t13} further extends this map along a natural embedding ${\sf Klein} \to {\sf SO}(3)$, the fundamental group of which can be identified as projection: $\ZZ \times \ZZ_{/2\ZZ} \xra{\sf proj} \ZZ_{/2\ZZ}$.}
This diagram witnesses an identification $\alpha^{\circ 2} \simeq \id_\cR$.
Using that $\cR$ is a braided-monoidal $(1,1)$-category, this identification canonically extends as a coherent factorization of $\ZZ \xra{\lag \alpha\rag} \Aut_{\Alg_2(\Cat_{(1,1)})}(\cR)$ through $\ZZ_{/2\ZZ} \to \Aut_{\Alg_2(\Cat_{(1,1)})}(\cR)$.\footnote{Further, there is an identification $\rho_{\vec{1}}^{\circ 2} \simeq \rho_{-\vec{1}} \circ \rho_{\vec{1}} = \id_\cR$.}

\end{itemize}

\end{enumerate}

\begin{notation}
We use the following notation in this paper:
\begin{itemize}

\item
For $V\in \cV$ an object in an $\infty$-category, and for $X\in \Spaces$ a space, the \bit{tensor of $V$ with $X$} is the object 
$
X 
\odot
V
:=
\colim
\Bigl(
X
\xra{!}
\ast
\xra{\lag V \rag}
\cV
\Bigr)
\in \cV
~.
$

\item 
$\Alg(\cV)$ is the $\oo$-category of associative algebras in a monoidal $\oo$-category $\cV$.

\item 
$\CAlg(\cV)$ is the $\oo$-category of commutative algebras in a symmetric monoidal $\oo$-category $\cV$.

\item 
$\Alg_{n-1}(\cV)$ is the $\oo$-category of $\cE_{n-1}$-algebras in a symmetric monoidal $\oo$-category $\cV$.

\item 
For $d \geq 1$, the $\infty$-category $\Cat_{(\infty,d)}$ is that of $(\infty,d)$-categories.

\end{itemize}

\end{notation}

\section{The basic action on higher braided-monoidal $(\infty,1)$-categories}

\subsection{Higher braided-monoidal categories}

Recall the topological operad of little $(n-1)$-disks, a version of which was introduced in \cite{bv}.
This topological operad presents an $\infty$-operad $\cE_{n-1}$.
For $\cU$ a symmetric monoidal $\infty$-category, 
\[
\Alg_{{n-1}}(\cU)
\]
is the $\infty$-category of \bit{$\cE_{n-1}$-algebras in $\cU$}.
Informally, an $\cE_{n-1}$-algebra in $\cU$ is an object $U\in \cU$ together with a $\SS^{n-2} \simeq \conf_2(\RR^{n-1})$-family and $\Sigma_2$-invariant of binary morphisms,
\[
\mu_2
\colon
\SS^{n-2}
\underset{\Sigma_2}
\odot
U^{\ot 2}
\longrightarrow
U
~,\qquad
\text{ or equivalently }
\qquad
\mu_2
\colon
\SS^{n-2}
\xra{~\Sigma_2 \text{\rm -equivariant}~}
\Hom_{\cU}( U^{\ot 2} , U)
~,
\]
which is associative in an appropriately homotopy coherent sense.

\begin{example}
In the case that $n=2$, such a binary morphism is familiar:
\[
\mu_2
\colon
U^{\ot 2}
\simeq
\SS^{0}
\underset{\Sigma_2}
\odot
U^{\ot 2}
\longrightarrow
U
~.
\]
In fact, there is a canonical equivalence between $\infty$-operads $\cE_1 \simeq \Ass$, and therefore an equivalence between $\infty$-categories of algebras over these $\infty$-operads:
\[
\Alg(\cU)
~:=~
\Alg_{\Ass}(\cU)
~\simeq~
\Alg_{1}(\cU)
~.
\]

\end{example}

\begin{example}
Let $0 \leq g < n-2$.
Suppose the underlying $(\infty,1)$-category of $\cU$ is in fact a $(g+1,1)$-category.
In this case, the space $\Hom_{\cU}( U^{\ot 2} , U)$ is a $g$-type.
Because $\SS^{n-2}$ is $(n-2)$-connective, such a binary morphism uniquely factors:
\[
\xymatrix{
\SS^{n-2}
\ar[rr]^-{\mu_2}
\ar[dr]_-{!}
&&
\Hom_{\cU}( U^{\ot 2} , U)
\\
&
\ast
\ar@{-->}[ur]_-{\ov{\mu}_2}
&
~.
}
\]
So the binary morphism takes the form of a $\Sigma_2$-invariant morphism $\ov{\mu}_2\colon U^{\ot 2} \to U$.
In fact\footnote{The unique morphism between symmetric $\infty$-operads $\cE_{n-1} \to {\sf Com}$ to the commutative operad witnesses an arity-wise $(n-3)$-truncation: $(\cE_{n-1})_{\leq n-3} \xra{\simeq} {\sf Com}$}, there is an equivalence between $\infty$-categories of algebras over these $\infty$-operads:
\[
\CAlg(\cU)
~:=~
\Alg_{{\sf Com}}(\cU)
~\simeq~
\Alg_{{n-2}}(\cU)
~.
\]

\end{example}

\begin{observation}
\label{t14}
\begin{enumerate}
\item[]

\item
Each linear injection $\RR \to \RR^{n-1}$ determines a morphism between topological operads $\cE_1 \to \cE_{n-1}$ of little disks, and therefore a morphism between $\infty$-operads:
\begin{equation}
\label{e24}
\SS^{n-2}
\xra{~\simeq~}
{\sf Hom}^{\sf lin.inj}(\RR , \RR^{n-1})
\longrightarrow
\Hom_{{\sf Operads}_{\infty}}(\cE_1 , \cE_{n-1})
~.
\end{equation}

\item
Let $\cU$ be a symmetric monoidal $\infty$-category.
The map~(\ref{e24}) corepresents a conservative functor between $\infty$-categories,
\begin{equation}
\label{e25}
\Alg_{{n-1}}(\cU)
\longrightarrow
\Alg(\cU)
~,\qquad
A
\mapsto (A , \overset{v} \ot)
~,
\end{equation}
which is evidently functorial in the symmetric monoidal $\infty$-category $\cU$.

\end{enumerate}

\end{observation}

Recall from~\cite{rezk} that an \bit{$(\infty,1)$-category} is a simplicial space $\cC\colon \bDelta^{\op} \to \Spaces$ that satisfies the following conditions.
\begin{itemize}
\item
The Segal condition: 
for each $[p]\in \bDelta$, the functor $\cC$ carries the opposite of the diagram 
\begin{equation}
\label{e51}
\xymatrix{
\{1\}
\ar[rr]
\ar[d]
&&
\{1<\cdots<p\}
\ar[d]
\\
\{0<1\}
\ar[rr]
&&
[p]
}
\end{equation}
in $\bDelta$
to a limit diagram in $\Spaces$.

\item
The univalent-completeness: the functor $\cC$ carries the opposite of the diagram
\begin{equation}
\label{e52}
\xymatrix{
\{0<2\}  \ar[rr]  \ar[d]
&&
\{0<1<2<3\}  \ar[d]
&&
\{1<3\} \ar[ll]  \ar[d]
\\
\ast \ar[rr]
&&
\ast
&&
\ast \ar[ll]
}
\end{equation}
in $\bDelta$
to a limit diagram in $\Spaces$.\footnote{Equivalently (by \cite{rezk.n}), the natural map
\[
\cC([0])\longrightarrow \cC([1])
\]
is an equivalence between the space of objects of $\cC$ and the subspace of the 1-morphisms $\cC([1])$ consisting of the isomorphisms.}
\end{itemize}

The $\infty$-category of \bit{$(\infty,1)$-categories} is the full $\infty$-subcategory
\begin{equation}
\label{e33}
\Cat_{(\infty,1)}
~\subset~
\PShv(\bDelta)
\end{equation}
consisting of the $(\infty,1)$-categories.
Note that the conditions of a simplicial space being an $(\infty,1)$-category are closed under the formation of limits.
In other words, $\Cat_{(\infty,1)}$ admits limits, and the functor~(\ref{e33}) preserves limits.

\begin{convention}
\label{r8}
As $\Cat_{(\infty,1)}$ admits finite products, it admits its Cartesian symmetric monoidal structure.
We regard $\Cat_{(\infty,1)}$ as a symmetric monoidal $\infty$-category with this symmetric monoidal structure.

\end{convention}

\begin{terminology}
\label{d4}
The $\infty$-category of \bit{$(n-1)$-monoidal $(\infty,1)$-categories} is
$\Alg_{{n-1}}(\Cat_{(\infty,1)})$.

\end{terminology}

\begin{definition}\label{def.duals}
\begin{enumerate}

\item[]

\item
Let $\cR\in \Alg(\Cat_{(\infty,1)})$ be a monoidal $(\infty,1)$-category.
An object $V \in \cR$ has a \bit{right dual} if left-tensoring with $V$
\[
\cR\xra{V\ot -}\cR
\]
has a right adjoint
\[
\cR\xla{-\ot V^R}\cR
\]
represented by right-tensoring with an object $V^R\in \cR$. 
The monoidal $(\infty,1)$-category $\cR$ has \bit{right duals} 
if every object of $\cR$ has a right dual.

\item
The $\infty$-category of \bit{rigid $(n-1)$-monoidal $(\infty,1)$-categories} is the full $\infty$-subcategory
\[
\Alg^{\rigid}_{{n-1}}( \Cat_{(\infty,1)} )
~\subset~
\Alg_{{n-1}}( \Cat_{(\infty,1)} )
\]
consisting of those $(n-1)$-monoidal $(\infty,1)$-categories $\cR$ for which, for each $v\in \SS^{n-2}$, the monoidal $(\infty,1)$-category $(\cR , \overset{v}\ot)$ (see the notation of~(\ref{e25})), has right duals.  

\end{enumerate}

\end{definition}

\begin{remark}
Adjoints are essentially unique, if they exist. 
Consequently, if an object $V\in\cR$ in a monoidal $(\infty,1)$-category has a right dual, it is essentially unique.
Better, provided $\cR$ is rigid, taking right duals defines an equivalence between monoidal $(\infty,1)$-categories:
\[
\cR
\longrightarrow
\cR^{{\sf rev} , \op_1}
\]
from $\cR$ to the monoidal opposite of its categorical opposite.
\end{remark}

\begin{observation}
\label{t15}
A monoidal $(\infty,1)$-category $\cR$ is rigid if and only if each object in $\cR$ has both a right dual and a left dual.

\end{observation}

\begin{observation}
\label{t22}
The property of an $(n-1)$-monoidal $(\infty,1)$-category being rigid can be detected on any of its underlying monoidal $(\infty,1)$-categories.
More precisely,
for each $v\in \SS^{n-2}$, the evident diagram among $\infty$-categories is a pullback:
\[
\xymatrix{
\Alg^{\rigid}_{{n-1}}( \Cat_{(\infty,1)} )
\ar[rr]^-{\rm fully~faithful}
\ar[d]_-{(\ref{e25})_v}
&&
\Alg_{{n-1}}( \Cat_{(\infty,1)} )
\ar[d]^-{(\ref{e25})_v}
\\
\Alg^{\rigid}( \Cat_{(\infty,1)} )
\ar[rr]^-{\rm fully~faithful}
&&
\Alg( \Cat_{(\infty,1)} )
.
}
\]

\end{observation}

\subsection{The basic action}

\begin{observation}
\label{t5}
Precomposing by the involution $\sO(1)\lacts \bDelta$ given by reversing linear orders defines an involution $\sO(1) \lacts \PShv(\bDelta)$.
This involution preserves the condition of a simplicial space being an $(\infty,1)$-category.
In this way, there is an involution,
\[
\sO(1)
~\lacts~
\Cat_{(\infty,1)}
~,\qquad
(-1)\cdot \cC~:=~ \cC^{\op_1}
~,
\]
on the $\infty$-category of $(\infty,1)$-categories given by \bit{taking 1-categorical opposites}.

\end{observation}

\begin{observation}
\label{t2}
Presenting the $\infty$-operad $\cE_{n-1}$ as the topological operad of little $(n-1)$-disks~(\cite{May}, \cite{bv}), the linear action $\sO(n-1)\lacts \RR^{n-1}$ defines an $\infty$-operadic action $\sO(n-1)\lacts \cE_{n-1}$:
\begin{equation}
\label{e22}
\BO(n-1)
\xra{~\bigl \lag \sO(n-1) \lacts \cE_{n-1} \bigr \rag~}
{\sf Operads}_{\infty}
~.
\end{equation}
The action~(\ref{e22}) corepresents an action $\sO(n-1)\lacts \Alg_{{n-1}}(\cU)$ on the $\infty$-category of $\cE_{n-1}$-algebras in any symmetric monoidal $\infty$-category $\cU$, functorially in the symmetric monoidal $\infty$-category $\cU$.

\end{observation}

\begin{remark}
\label{r3}
We offer a way to conceptualize the action of Observation~\ref{t2}.
\begin{enumerate}

\item
Each rank $(n-1)$ vector bundle $\xi = (E \to B)$ determines a composite functor 
\[
B \xra{~\xi~} \BO(n-1) \xra{~(\ref{e22})~} {\sf Operads}_{\infty}
~.
\]
In particular, each $(n-1)$-dimensional vector space $V$ over $\RR$ determines an $\infty$-operad $\cE_V$ that is non-canonically equivalent with $\cE_{n-1} =: \cE_{\RR^{n-1}} \simeq \cE_V$.

\item
For each $v\in \SS^{n-2}$, 
there is a standard linear isomorphism $\sT_v \SS^{n-2} \oplus  \RR \cong {\sf Span}\{v\}^{\perp} \oplus   {\sf Span}\{v\} \cong \RR^{n-1}$.
This linear isomorphism from a direct sum determines a morphism between $\infty$-operads, from the $\infty$-operadic tensor product: 
$\cE_{\sT_v \SS^{n-2}} \ot \cE_1 \to \cE_{n-1}$.  
Dunn's additivity (\cite{dunn},~\cite{HA}) grants that this morphism is an equivalence.
Using the definition of the $\infty$-operadic tensor product, this equivalence between $\infty$-operads corepresents an equivalence between $\infty$-categories,
\begin{equation}
\label{e23}
\Alg_{{n-1}}(\cU)
~\simeq~
\Alg_{{\sT_v \SS^{n-2}}}\bigl(
\Alg
(\cU)
\bigr)
~,
\end{equation}
which is evidently functorial in the symmetric monoidal $\infty$-category $\cU$.

\item
Let $v\in \SS^{n-2}$.
Reflection in the line ${\sf Span}\{v\} \subset \RR^{n-1}$ defines an action $\sO(1) \xra{\lag \refl_v \rag} \sO(n-1) \lacts \cE_{n-1}$, which corepresents the involution
\[
(-)^{{\sf rev}_v}
\colon
\Alg_{{n-1}}(\cU)
~\underset{(\ref{e23})}\simeq~
\Alg_{{\sT_v \SS^{n-2}}}\bigl(
\Alg
( \Cat_{(\infty,1)} )
\bigr)
\xra{
\Alg_{{\sT_v \SS^{n-2}}}( (-)^{{\sf rev}}) 
}
\Alg_{{\sT_v \SS^{n-2}}}\bigl(
\Alg
( \Cat_{(\infty,1)} )
\bigr)
~\underset{(\ref{e23})}\simeq~
\Alg_{{n-1}}(\cU)
~,
\]
that takes monoidal opposites in the $v$-direction.

\end{enumerate}

\end{remark}

\begin{example}
As the case that $n=3$, the action $\sO(2) \lacts \Alg_{2}(\cU)$ determines, for each $\cE_2$-algebra $A$ in $\cU$, a automorphism $\alpha \colon \ZZ \simeq \Omega \sO(2) \lacts A$ of the $\cE_2$-algebra.  
On the underlying object of $A$, this automorphism $\alpha(1) = \id$ is the identity.
On the family of 2-ary morphisms, 
$(\SS^1 \underset{\Sigma_2}\odot A^{\ot 2} \xra{\mu_2}A)\in \Map^{\Sigma_2}\bigl( \SS^1 , \Hom_{\cU}( A^{\ot 2} , A) \bigr)$, it is the element
\[
\alpha(2)
\in 
\Aut_{\Map^{\Sigma_2}\bigl( \SS^1 , \Hom_{\cU}(
A^{\ot 2}
,
A)
\bigr)
}
(\mu_2)
~\simeq~
\Omega_{\mu_2} \Map^{\Sigma_2}\bigl( \SS^1 , \Hom_{\cU}(
A^{\ot 2}
,
A)
\bigr)
\]
corresponding to the based loop at $\mu_2$ adjoint to 
\[
\SS^1 \odot \bigl( \SS^1 \underset{\Sigma_2}\odot A^{\ot 2} \bigr)
~\simeq~
\bigl( \SS^1 \times \SS^1 \bigr) \underset{\Sigma_2}\odot A^{\ot 2}
\xra{~( \SS^1-{\rm multiply} ) ~\odot~ \id~}
\SS^1 \underset{\Sigma_2}\odot A^{\ot 2}
\xra{\mu_1}
A
~.
\]
In other words, $\alpha$ is the automorphism of the braiding structure of $A$ that carries, for each $a,b\in A$, the equivalence $a \ot b \xra{\beta_{a,b}} b\ot a$ in $A$ to the equivalence $a \ot b \xra{ \beta_{b}^a \circ (\beta_{a}^b \circ \beta_{b}^a)} b \ot a$ in $A$.

\end{example}

\begin{observation}
\label{t25}
Let $\cU$ be a symmetric monoidal $\infty$-category.
Because the action of Observation~\ref{t5} preserves products, 
it and the action of Observation~\ref{t2} assemble as an action 
\begin{equation}
\label{e10}
\sO(n-1)\times \sO(1) 
~\lacts~
\Alg_{{n-1}}(\cU)
~.
\end{equation}

\end{observation}

\begin{observation}
\label{t16}
Let $\cR\in \Alg_{{n-1}}( \Cat_{(\infty,1)} )$ be an $(n-1)$-monoidal $(\infty,1)$-category.
Note that $\cR$ is rigid if and only if its 1-categorical opposite $\cR^{\op_1}$ is rigid if and only if its underlying monoidal opposite $\cR^{{\sf rev}}$ is rigid.  
Using that the standard homomorphism $\sO(1) \times \sO(1) \to \sO(n-1)\times \sO(1)$ is surjective on path-components, it follows that the action~(\ref{e10}) restricts as an action
\begin{equation}
\label{e11}
\sO(n-1)\times \sO(1)
~\lacts~
\Alg^{\rigid}_{{n-1}}( \Cat_{(\infty,1)} )
\end{equation}
on the $\infty$-category of $(n-1)$-monoidal $(\infty,1)$-categories with duals.

\end{observation}

\section{Projective colimits}

\subsection{A universal null-homotopy}

In this section, we prove the following.
\begin{lemma}
\label{t27}
There is a canonically commutative diagram among pointed spaces:
\begin{equation}
\label{e34}
\xymatrix{
\RR\PP^{n-1}
\ar[d]
\ar[rr]
&&
\RP^{n-2}_+\wedge\BO(1)
\ar[d]
\\
\ast
\ar[rr]
&&
\SS^{n-2}_+ \underset{\sO(1)}\wedge {\sf BU}(1)
~.
}
\end{equation}
\end{lemma}

Our strategy is to obtain this square diagram~(\ref{e34}) as a colimit indexed by $\RR\PP^{n-2}$ of square diagrams.
Specifically, we obtain~(\ref{e34}) through the following sequence of functors between $\infty$-categories from a square diagram in the $\infty$-category ${\sf CGroups}$ of commutative continuous groups:
\begin{eqnarray}
\label{e36}
{\sf CGroups}
&
\xra{~\rm forget~}
&
\Mod_{\sO(1)}({\sf Groups})
\\
\nonumber
&
\xra{~\Mod_{\sO(1)}(\sB)~}
&
\Mod_{\sO(1)}( \Spaces^{\ast/})
~=~
\Fun\bigl(
\BO(1)
,
\Spaces^{\ast/}
\bigr)
\\
\nonumber
&
\xra{
\gamma_1(n-1)^\ast
}
&
\Fun\bigl(
\RR\PP^{n-2}
,
\Spaces^{\ast/}
\bigr)
\\
\nonumber
&
\xra{~\colim~}
&
\Spaces^{\ast/}
~.
\end{eqnarray}
The first forgetful functor assigns to each commutative group its underlying group as it is endowed with an involution given by taking inverses: 
$G\mapsto \bigl( \sO(1) \underset{(-)^{-1}}\lacts G \bigr)$.
The second functor is implemented by taking classifying spaces: 
$\bigl( \sO(1) \lacts G \bigr) \mapsto \bigl( \sO(1) \lacts \sB G\bigr)$.
The third functor is restriction along the functor $\RR\PP^{n-2} \xra{\gamma_1(n-1)} \BO(1)$ classifying the universal line bundle over $\RR\PP^{n-2}$.  
The last functor is given by taking colimits of $\RR\PP^{n-2}$-indexed diagrams.

\begin{observation}
\label{t28}
The path of homomorphisms
\[
[0,1]
\longrightarrow
{\sf Hom}\bigl( \ZZ,\sU(1) \bigr)
~,\qquad
t
\mapsto
(~
n
\mapsto
e^{\pi i n t}
~)
~,
\]
presents an invertible 2-cell witnessing a commutative diagram among commutative continuous groups:
\begin{equation}
\label{e35}
\xymatrix{
\ZZ
\ar[d]\ar[rr]^-{\lag \rm generator \rag}
&& 
\sO(1)
\ar[d]^-{\pm 1 \longmapsto \pm 1}
\\
\ast
\ar[rr]
&&
\sU(1)
.
}
\end{equation}

\end{observation}

\begin{proof}[Proof of Lemma~\ref{t27}]

The composite functor~(\ref{e36}) carries the diagram~(\ref{e35}) in ${\sf CGroups}$ to the following diagram in $\Spaces^{\ast/}$:
\begin{equation}
\label{e37}
\xymatrix{
\colim\Bigl(\RR\PP^{n-2}\xra{\gamma} \BO(1) \xra{\lag \sO(1)\underset{\sB(-)^{-1}}\lacts\sB\ZZ  \rag}
\Spaces^{\ast/}
\Bigr)
\ar[r]\ar[d]
&
\colim\Bigl(\RR\PP^{n-2}\xra{\gamma} \BO(1) \xra{ \lag \sO(1)\underset{\sB(-)^{-1}}\lacts\BO(1) \rag}
\Spaces^{\ast/}
\Bigr)
\ar[d]
\\
\colim\Bigl(\RR\PP^{n-2}\xra{\gamma} \BO(1) \xra{ \lag \sO(1)\underset{\sB(-)^{-1}}\lacts \ast \rag}
\Spaces^{\ast/}
\Bigr)
\ar[r]&
\colim\Bigl(\RR\PP^{n-2}\xra{\gamma} \BO(1) \xra{ \lag \sO(1)\underset{\sB(-)^{-1}}\lacts {\sf BU}(1) \rag }
\Spaces^{\ast/}
\Bigr)
.
}
\end{equation}
The proof is complete upon identifying each term in this square of pointed spaces as those in the statement of the lemma.
\begin{itemize}
\item[{\bf Bottom left.}]
The bottom left term is a colimit of a constant functor at the initial object $\ast \in \Spaces^{\ast/}$.
Therefore, the bottom left term uniquely identified as the initial pointed space:
\[
\colim\Bigl(\RR\PP^{n-2}\xra{\gamma} \BO(1) \xra{\sO(1)\underset{\sB(-)^{-1}}\lacts \ast }
\Spaces^{\ast/}
\Bigr)
~\simeq~
\ast
~.
\]

\item[{\bf Top right.}]
Because each element in $\sO(1)$ has order 2, the involution $\sO(1) \lacts \sO(1)$ given by taking inverses is trivial.
Therefore, the involution $\sO(1) \lacts \BO(1)$ is trivial.
Therefore, the top right term is a constant functor at the pointed space $\BO(1)$. 
Therefore, the top right term is the tensor in pointed spaces,
\[
\colim\Bigl(\RR\PP^{n-2}\xra{\gamma} \BO(1) \xra{\sO(1)\underset{\sB(-)^{-1}}\lacts\BO(1) }
\Spaces^{\ast/}
\Bigr)
~\simeq~
\RR\PP^{n-2}
\odot
\BO(1)
~\simeq~
\RR\PP^{n-2}_+
\wedge
\BO(1)
~,
\]
which is given by smash product with $\RR\PP^{n-2}_+$.

\item[{\bf Top left.}]
Recall the standard identification as pointed spaces, $\sB \ZZ \simeq \SS^1 \simeq (\RR^1)^+$, where the last pointed space is the 1-point compactification.  
This identification extends as a standard identification between the functor
$\BO(1) \xra{\lag \sO(1) \underset{\sB(-)^{-1}} \lacts \sB \ZZ \rag} \Spaces^{\ast/}$ and the functor $\BO(1) \xra{L\mapsto L^+} \Spaces^{\ast/}$ whose value on a line is its 1-point compactification.
By definition, the \bit{Thom space} of a line bundle $\xi = (E \to B)$ is the pointed space ${\sf Thom}(\xi) := \colim\bigl( B \xra{\xi} \BO(1) \xra{L\mapsto L^+} \Spaces^{\ast/})$.
Therefore, the top left term can be identified as the Thom space of the universal line bundle over $\RR\PP^{n-2}$, which is $\RR\PP^{n-1}$:
\[
\colim\Bigl(\RR\PP^{n-2}\xra{\gamma} \BO(1) \xra{\sO(1)\underset{\sB(-)^{-1}}\lacts\sB\ZZ }
\Spaces^{\ast/}
\Bigr)
~\simeq~
{\sf Thom}\bigl(\gamma_1(n-1)\bigr)
~\simeq~
\RR\PP^{n-1}
~.
\]

\item[{\bf Bottom right.}]
We explain the following sequences of equivalences among pointed spaces:
\begin{eqnarray*}
\colim\Bigl(\RR\PP^{n-2}\xra{\gamma} \BO(1) \xra{ \lag \sO(1)\underset{\sB(-)^{-1}}\lacts {\sf BU}(1) \rag }
\Spaces^{\ast/}
\Bigr)
&
~\simeq~
&
\\
\colim\Bigl(\SS^{n-2} \xra{\sO(1) \text{\rm -quotient}} \RR\PP^{n-2}\xra{\gamma} \BO(1) \xra{ \lag \sO(1)\underset{\sB(-)^{-1}}\lacts {\sf BU}(1) \rag }
\Spaces^{\ast/}
\Bigr)_{\sO(1)}
&
~\simeq~
&
\\
\colim\Bigl(\SS^{n-2} \to \ast \to \BO(1) \xra{ \lag \sO(1)\underset{\sB(-)^{-1}}\lacts {\sf BU}(1) \rag }
\Spaces^{\ast/}
\Bigr)_{\sO(1)}
&
~\simeq~
&
\\
\colim\Bigl(\SS^{n-2} \to \ast \xra{ \lag {\sf BU}(1) \rag }
\Spaces^{\ast/}
\Bigr)_{\sO(1)}
&
~\simeq~
&
\Bigl(\SS^{n-2} \odot {\sf BU}(1)
\Bigr)_{\sO(1)}
\\
\Bigl(\SS^{n-2}_+ \wedge {\sf BU}(1)
\Bigr)_{\sO(1)}
&
~=~
&
\SS^{n-2}_+
\underset{\sO(1)}\wedge
{\sf BU}(1)
~.
\end{eqnarray*}
The first equivalence is an instance of a rather general situation, as we explain.
Let $\cK\xra{\w{F}} \cS$ be a functor between $\infty$-categories for which the colimit $\colim(\w{F})\in \cS$ exists.
Direct from the universal property of colimits, an action by a continuous group $G\lacts \cK$ determines an action $G \lacts \colim(\w{F})$.
Comparing universal properties of colimits and $G$-coinvariants reveals that a factorization $\w{F}\colon \cK \xra{G \text{\rm -quotient}} \cK_G \xra{F} \cS$ through the coinvariants determines an identification of the coinvariants of this action on the colimit:
\[
\colim(\w{F})_G
\xra{~\simeq~}
\colim(F)
~
\in \cS
~.
\]
The first equivalence follows by applying this to the situation in which 
\begin{itemize}
\item
$( G \lacts \cK ) = \bigl( \sO(1) \lacts \SS^{n-2} \bigr)$ is the sphere with its antipodal action, 

\item
$( \cK  \xra{\w{F}} \cS) = \bigl( \SS^{n-2} \to  \RR\PP^{n-2} \xra{\gamma} \BO(1) \xra{ \lag \sO(1) \lacts {\sf BU}(1) \rag} \Spaces^{\ast/}\bigr)$,

\item 
$( \cK_G \xra{F} \cS) = \bigl( {\SS^{n-2}}_{\sO(1)} \simeq \RR\PP^{n-2} \xra{\gamma} \BO(1) \xra{ \lag \sO(1) \lacts {\sf BU}(1) \rag} \Spaces^{\ast/}\bigr)$.

\end{itemize}
The second equivalence follows from the commutativity of the diagram among spaces:
\[
\xymatrix{
\SS^{n-2}
\ar[rr]
\ar[d]_-{\sO(1) \text{\rm -quotient}}
&&
\ast
\ar[d]
\\
\RR\PP^{n-2}
\ar[rr]^-{\gamma}
&&
\BO(1)
.
}
\]
The third equivalence simply recognizes the constant functor at ${\sf BU}(1)$.
The fourth equivalence recognizes colimits of constant functors as tensoring with classifying spaces of the indexing $\infty$-category, in this case $\SS^{n-2}$.
The fifth equivalence recognizes such tensoring in $\Spaces^{\ast/}$ in terms of smash product.  
The final equivalence follows upon recognizing the $\sO(1)$-action on $\SS^{n-2}_+\wedge {\sf BU}(1)$ as the diagonal $\sO(1)$-action from each factor.

\end{itemize}

\end{proof}

\begin{observation}
\label{t38}
By construction, the pointed map
$
\RR\PP^{n-1}
\to
\RR\PP^{n-2}_+ \wedge \BO(1)
$
in~(\ref{e34}) is the colimit of the $\RR\PP^{n-2}$-indexed diagram of pointed maps:
\[
\sB \ZZ
~\simeq~
L^+
~\simeq~
\PP(L\oplus \RR^{\{n\}})
\xra{~
(L' \subset L\oplus \RR^{\{n\}})
\mapsto
L'
~}
\RP^\infty
~=~
\BO(1)
\qquad
\Bigl(
~
(L \subset \RR^{n-1})\in \RR\PP^{n-2}
~
\Bigr)
~,
\]
which is $\sB$ applied to the homomorphism $\ZZ \xra{\lag -1\rag} \sO(1)$.

\end{observation}

For $L\subset \RR^{n-1}$ a 1-dimensional vector subspace, consider the orthogonal transformation 
that restricts as reflection in the line $L$ and as the identity in the orthogonal complement $L^{\perp}$: 
\[
{\sf Ref}_L\colon 
\RR^{n-1}
~\cong~ 
L\oplus L^\perp 
\xra{~(-1) \oplus {\sf id}_{L^\perp}~} 
L\oplus L^\perp 
~\cong~ 
\RR^{n-1}
~.
\]
This orthogonal transformation has order 2 and is continuous in the argument $L$, thereby defining a map
\begin{equation}
\label{e40}
\RR\PP^{n-2}
\longrightarrow
{\sf Homo}\bigl(
\sO(1)
,
\sO(n-1)
\times
\sO(1)
\bigr)
~,\qquad
L\longmapsto 
\bigl( -1 \mapsto  ( {\sf Ref}_L , -1 ) \bigr)
~.
\end{equation}

\begin{observation}
\label{t29}
Through the equivalences between spaces of morphisms given, respectively, by the Tensor-Hom-adjunction and the $(\sB,\Omega)$-adjunction,
\begin{eqnarray*}
\Hom_{\Spaces^{\ast/}}\bigl( \RR\PP^{n-2}_+ \wedge \BO(1)  , \BO(n-1) \times \BO(1) \bigr)
&
~\simeq~
&
\\
\Map\Bigl(
\RR\PP^{n-2}
,
\Hom_{\Spaces^{\ast/}}\bigl(\BO(1)  , \BO(n-1) \times \BO(1) \bigr)
\Bigr)
&
~\simeq~
&
\\
\Map\Bigl(
\RR\PP^{n-2}
,
\Hom_{{\sf Groups}}\bigl(\sO(1)  , \sO(n-1) \times \sO(1) \bigr)
\Bigr)
&
&
~,
\end{eqnarray*}
the map~(\ref{e40}) corresponds to the map
\begin{equation}
\label{e39}
\RR\PP^{n-2}_+
\wedge
\BO(1)
\longrightarrow
\BO(n-1) \times \BO(1)
~,\qquad
( L\subset \RR^{n-1} , L')
\longmapsto 
( L^\perp \oplus L' , L' )
~.
\end{equation}
\end{observation}

Observations~\ref{t38} and~\ref{t29} lend the following.
\begin{observation}
\label{t39}
By construction, the composite pointed map $\RP^{n-1} \xra{(\ref{e34})} \RP^{n-2}_+ \wedge \BO(1) \xra{(\ref{e39})} \BO(n-1)\times \BO(1)$ is the $\RP^{n-2}$-indexed diagram of pointed maps
\[
\sB \ZZ
~\simeq~
L^+
~\simeq~
\PP(L\oplus \RR^{\{n\}})
\longrightarrow
\RP^\infty
\longrightarrow
\BO(n-1)\times \BO(1)
\qquad
\Bigl(
~
(L \subset \RR^{n-1})\in \RR\PP^{n-2}
~
\Bigr)
~,
\]
\[
(L' \subset L\oplus \RR^{\{n\}})
\longmapsto
\bigl(
L^\perp \oplus L'
,
L'
\bigr)
~,
\]
which is $\sB$ applied to the homomorphism $\ZZ \xra{\lag( {\sf Ref}_L , -1) \rag } \sO(n-1)\times \sO(1)$.

\end{observation}

Observation~\ref{t39} lends the following.
\begin{observation}
\label{t30}
The composite pointed map fits into the standard pullback among spaces:
\begin{equation}
\label{e41}
\xymatrix{
\sO(n)
\ar[rrrr]
\ar[d]_-{\sO(n-1)\times \sO(1) \text{\rm -quotient}}
&&
&&
\ast
\ar[d]
\\
\RR\PP^{n-1}
\ar[rr]^-{ (\ref{e34}) }
&&
\RR\PP^{n-2}_+ \wedge \BO(1)
\ar[rr]^-{ (\ref{e39}) }
&&
\BO(n-1) \times \BO(1)
.
}
\end{equation}

\end{observation}

\subsection{The $n=2$ case in projective families}

\begin{lemma}
\label{prop.key}
There is a canonical filler in the diagram among pointed spaces:
\begin{equation}
\label{e12}
\xymatrix{
\RR\PP^{n-2}_+ \wedge \BO(1)
\ar[d]_-{ (\ref{e34}) }
\ar[rr]^-{ (\ref{e39}) }
&&
\BO(n-1)\times \BO(1)
\ar[d]^-{\Bigl\lag \sO(n-1)\times \sO(1)
\underset{(\ref{e11})} \lacts
\Alg^{\rigid}_{{n-1}}( \Cat_{(\infty,1)} )
\Bigr\rag}
\\
\SS^{n-2}_+ \underset{\sO(1)} \wedge {\sf BU}(1)
\ar@{-->}[rr]
&&
{\sf BAut}_{\CAT}\Bigl(
\Alg_{{n-1}}^{\rigid}( \Cat_{(\infty,1)} )
\Bigr)
.
}
\end{equation}

\end{lemma}

\begin{proof}

We first explain the commutative diagram among $\sO(1)$-pointed spaces:
\begin{equation}
\label{e46}
\xymatrix{
\BO(1)
\ar[d]_-{\pm 1 \mapsto \pm 1}
\ar[rr]^-{\sO(1) \text{\rm - invariant}}
&&
\BO(1)\times \BO(1)
\ar[d]^-{\oplus}
\\
{\sf BU}(1)
\ar[rr]^-{\sO(1) \text{\rm - invariant}}
&&
\BO(2)
.
}
\end{equation}
Regard the two terms on the right as $\sO(1)$-pointed spaces with the trivial $\sO(1)$-action;
regard the two spaces on the left as $\sO(1)$-pointed spaces with the $\sO(1)$-action given by $\sB(-)^{-1}$.  
With these $\sO(1)$-actions, the vertical maps are evidently $\sO(1)$-equivariant.
As indicated, the horizontal arrows in fact witness (unpointed) $\sO(1)$-coinvariants:
\[
\BO(1)_{\sO(1)}
~\simeq~
\sB\bigl(
\sO(1) \rtimes \sO(1)
\bigr)
~\simeq~
\BO(1) \times \BO(1)
\]
and
\[
{\sf BU}(1)_{\sO(1)}
~\simeq~
\sB\bigl(
\sU(1) \rtimes \sO(1)
\bigr)
~\simeq~
\BO(2)
~.
\]
Commutativity of the diagram then follows.

We next explain the commutative diagram among pointed spaces:
\begin{equation}
\label{e21}
\xymatrix{
\BO(1)\times \BO(1)
\ar[rrrr]^-{\Bigl\lag \sO(1)\times \sO(1)
\underset{(\ref{e11})} \lacts
\Alg^{\rigid}_{{n-1}}( \Cat_{(\infty,1)} )
\Bigr\rag}
\ar[d]_-{\oplus}
&&
&&
{\sf BAut}_{\CAT}\Bigl(
\Alg^{\rigid}( \Cat_{(\infty,1)} )
\Bigr)
\\
\BO(2)
\ar@{-->}[rrrr]_-{\exists !}
\ar@{-->}[urrrr]^-{\rm Thm~\ref{t11}}
&&
&&
{\sf BAut}_{\CAlg(\CAT)}\Bigl(
\Alg^{\rigid}( \Cat_{(\infty,1)} )
\Bigr)
\ar[u]_-{\rm forget}^-{\simeq}
.
}
\end{equation}
The top horizontal pointed map selects the action of Observation~\ref{t16}.
The diagonal filler is Theorem~\ref{t11}, as indicated.
Now, recall that the symmetric monoidal structure on $\Cat_{(\infty,1)}$, and therefore on $\Alg^{\rigid}( \Cat_{(\infty,1)} )$, is Cartesian.
Therefore, each automorphism of $\Alg^{\rigid}( \Cat_{(\infty,1)} )$ preserves finite products.
It follows that each automorphism of the $\infty$-category $\Alg^{\rigid}( \Cat_{(\infty,1)} )$ uniquely extends as a symmetric monoidal automorphism.
This is to say that the right vertical forgetful functor is an equivalence.
We conclude the bottom horizontal filler.

We now establish the commutative square among $\Omega \RR\PP^{n-2}$-pointed spaces, which is to say, the functor $\RR\PP^{n-2} \to \Fun\bigl( c_1\times c_1 , \Spaces^{\ast/})$:
for $L\in \RR\PP^{n-2}$,
\begin{equation}
\label{e27}
\xymatrix{
\BO(L) \times \BO(1)
\ar[rr]^-{\sB\lag {\sf Ref}_L \rag \times {\sf id}}
\ar[d]_-{\Bigl\lag \sO(L)\times \sO(1)
\underset{(\ref{e11})} \lacts
\Alg_{{ }}^{\rigid}( \Cat_{(\infty,1)} )
\Bigr\rag}
&&
\BO(n-1) \times \BO(1)
\ar[d]^-{\Bigl\lag \sO(n-1)\times \sO(1)
\underset{(\ref{e11})} \lacts
\Alg_{{n-1}}^{\rigid}( \Cat_{(\infty,1)} )
\Bigr\rag}
\\
{\sf BAut}_{\CAlg(\CAT)}\Bigl(
\Alg_{{L}}^{\rigid}( \Cat_{(\infty,1)} ) 
\Bigr)
\ar[rr]_-{\Alg_{{L^{\perp}}}(-)}
&&
{\sf BAut}_{\CAT}\Bigl(
\Alg_{{n-1}}^{\rigid}( \Cat_{(\infty,1)} ) 
\Bigr)
.
}
\end{equation}
The lower left term in this diagram~(\ref{e27}) is the composite functor 
\[
\RR\PP^{n-2} 
\longrightarrow
\RR\PP^{\infty}
~\simeq~
\BO(1) \xra{~\bigl \lag \sO(1) \underset{\rm Obs~\ref{t2}}\lacts \Alg^{\rigid}(\Cat_{(\infty,1)}) \bigr \rag~} 
{\sf BAut}_{\CAlg(\CAT)}\Bigl(
\Alg_{{ }}^{\rigid}( \Cat_{(\infty,1)} ) 
\Bigr)
~.
\]
The each of the two right terms in this diagram~(\ref{e27}) is independent of $L\in \RR\PP^{n-2}$.
The right vertical pointed map classifies the action~(\ref{e11}), as indicated, and is independent of $L\in \RR\PP^{n-2}$.
The group $\sO(L)$ is that of orthogonal transormations of the 1-dimensional vector subspace $L\subset \RR^{n-1}$;
there is a unique isomorphism $\sO(L) \simeq \sO(1)$ between groups.
The left vertical map can be achieved through the filled diagram~(\ref{e21}).  
The top horizontal pointed map is the result of applying the functor ${\sf Groups} \xra{\sB} \Spaces^{\ast/}$ to the product of the homomorphisms $\sO(L) \xra{{\sf Ref}_L} \sO(n-1)$ and $\sO(1) \xra{ \id} \sO(1)$.
It remains to establish the bottom horizontal pointed map.
Dunn's additivity (\cite{dunn}, Theorem~5.1.2.2  of~\cite{HA}), grants that the morphism between $\infty$-operads, $\cE_{n-2} \ot \cE_{1} \xra{\simeq} \cE_{n-1}$, is an equivalence.
This equivalence is evidently equivariant with respect to the morphism $\sO(n-2)\times \sO(1) \xra{\oplus} \sO(n-1)$ between continuous groups.
The canonical commutative square among pointed spaces
\[
\xymatrix{
\RR\PP^{n-2} 
\ar[rr]
\ar[d]
&&
\ast
\ar[d]
\\
\BO(n-2)\times \BO(1)
\ar[rr]^-{\oplus}
&&
\BO(n-1)
~,
}
\]
then determines an equivalence from the composite functor 
\[
\Bigl(
\RR^{n-2} \to \BO(n-2)\times \BO(1) \xra{\bigl \lag \sO(n-2)\times \sO(1) \underset{\rm Obs~\ref{t2}}\lacts \Alg_{n-2}\bigl( \Alg_{{}}^{\rigid}(\Cat_{(\infty,1)})\bigr) \bigr\rag }  \CAT
\Bigr)
~\simeq~
\Bigl(
\RR\PP^{n-2} \xra{!}\ast \xra{\bigl \lag \Alg_{n-1}^{\rigid}(\Cat_{(\infty,1)}) \bigr\rag} \CAT
\Bigr)
~:
\]
\[
\Alg_{L^{\perp}}\bigl( \Alg_L^{\rigid}(\Cat_{(\infty,1)}) \bigr)
~\simeq~
\Alg_{L^{\perp} \oplus L}^{\rigid}(\Cat_{(\infty,1)})
~\simeq~
\Alg_{\RR^{n-1}}^{\rigid}(\Cat_{(\infty,1)})
~\simeq~
\Alg_{n-1}^{\rigid}(\Cat_{(\infty,1)})
~,\qquad
\bigl(
~L \in \RR\PP^{n-2}
~\bigr)
~.
\]
This explains the lower horizontal pointed map in the diagram~(\ref{e27}), as well as commutativity of the diagram~(\ref{e27}), as it depends on $L\in \RR\PP^{n-2}$.

Now, by way of the morphisms $\Omega \RR\PP^{n-2} \to \sO(1) \to \ast$ among continuous groups, regard each of the diagrams~(\ref{e46}),~(\ref{e21}), and~(\ref{e27}) as a commutative square among $\Omega \RR\PP^{n-2}$-pointed spaces.
Concatenating, from left to right, the commutative squares~(\ref{e46})$+$(\ref{e21})$+$(\ref{e27}) results in a commutative diagram of $\Omega \RR\PP^{n-2}$-pointed spaces:
\begin{equation}
\label{e47}
\xymatrix{
\BO(1) 
\ar[r]
\ar[d]
&
\BO(1) \times \BO(1)
\ar[r]
\ar[d]
&
\BO(1) \times \BO(1)
\ar[r]
\ar[d]
&
\BO(n-1) \times \BO(1)
\ar[d]
\\
{\sf BU}(1) 
\ar[r]
&
\BO(2)
\ar[r]
&
{\sf BAut}_{\CAlg(\CAT)}\Bigl(
\Alg^{\rigid}_{{ }}( \Cat_{(\infty,1)} ) 
\Bigr)
\ar[r]
&
{\sf BAut}_{\CAT}\Bigl(
\Alg_{{n-1}}^{\rigid}( \Cat_{(\infty,1)} ) 
\Bigr)
.
}
\end{equation}
Consider the outer square from~(\ref{e47}) of $\Omega \RR\PP^{n-2}$-pointed spaces.
By construction, the right vertical pointed map in~(\ref{e47}) is a trivial $\Omega \RR\PP^{n-2}$-pointed map.  
So the left vertical pointed map in~(\ref{e47}) canonically factors through its $\Omega \RR\PP^{n-2}$-coinvariants, which are identified in the proof of Lemma~\ref{t27}:
\begin{equation}
\label{e48}
\xymatrix{
\RR\PP^{n-2}_+ \wedge \BO(1)
\ar[rr]^-{\simeq}
\ar[d]_-{(\ref{e34})}
&&
\BO(1)_{\Omega \RR\PP^{n-2}}
\ar[rr]
\ar[d]
&&
\BO(n-1) \times \BO(1)
\ar[d]^-{\Bigl\lag \sO(n-1)\times \sO(1)
\underset{(\ref{e11})} \lacts
\Alg^{\rigid}_{{n-1}}( \Cat_{(\infty,1)} )
\Bigr\rag}
\\
\SS^{n-2}_+
\underset{\sO(1)} \wedge
{\sf BU}(1)
\ar[rr]^-{\simeq}
&&
{\sf BU}(1)_{\Omega \RR\PP^{n-2}}
\ar[rr]
&&
{\sf BAut}_{\CAT}\Bigl(
\Alg_{{n-1}}^{\rigid}( \Cat_{(\infty,1)} ) 
\Bigr)
.
}
\end{equation}
The sought commutative square among pointed spaces is the outer commutative square of the diagram~(\ref{e48}).

\end{proof}

\section{Proof of main results}
Here we prove the our main theorem (Theorem~\ref{t13}), which identifies natural symmetries of each rigid $(n-1)$-monoidal $(g,d)$-category.

\subsection{Proof of the Proposition~\ref{t10}}
Here we prove Proposition~\ref{t10}, restated below as Theorem~\ref{t19}.
We first establish the following characterization of equivariant extensions of action maps.

\begin{lemma}
\label{t18}
Let $H \xra{\rho} G$ be a morphism between continuous groups that is surjective on path-components.
Let $\cC$ be an $\infty$-category that admits colimits indexed by $\infty$-groupoids.  
Let $A \in \Mod_{H}(\cC)$ be an object in $\cC$ equipped with an $H$-action.
The following moduli spaces are canonically equivalent to one another.
\begin{enumerate}
\item
The moduli space of $H$-equivariant extensions of the action morphism:
\[
\xymatrix{
&
H
\odot 
A
\ar[dr]^-{\rm action}
\ar[dl]_-{\rho \times \id }
&
\\
G
\odot 
A
\ar@{-->}[rr]^-{H- \rm equivariant}
&
&
A
~.
}
\]

\item
The moduli space of trivializations of the restricted $\Omega G_H$-action on $A$:
\[
\xymatrix{
\Omega G_H
\ar[rrrr]^-{\Omega ( G_H \xra{!_H} \ast_H = \sB H) }
\ar[d]
&&
&&
H
\ar[rrrr]^-{ \bigl \lag H \lacts A \bigr\rag}
&&
&&
\Aut_{\cC}(A)
\\
\ast
\ar@{-->}[urrrrrrrr]
&&
&&
&&
&&
.
}
\]

\end{enumerate}

\end{lemma}

\begin{proof}
The statement follows by establishing the following sequence of equivalences among morphism spaces.
\begin{eqnarray*}
\Hom_{\Mod_H(\cC)^{H\odot A/}} \bigl( G \odot A , A \bigr)
&
\simeq
&
\Hom_{\Mod_H(\Spaces)^{H/}} \bigl( G , \End_\cC(A)\bigr)
\\
&
\simeq
&
\Hom_{\Mod_H(\Spaces)^{H/}} \bigl( G , \Aut_\cC(A)\bigr)
\\
&
\simeq
&
\Hom_{(\Spaces^{\ast/})_{/\sB H}} \bigl( G_H , \Aut_\cC(A)_H\bigr)
\\
&
\simeq
&
\Hom_{(\Spaces^{\ast/})_{/{\sf BAut}_\cC(A)}} \bigl( G_H , \ast \bigr)
\\
&
\simeq
&
\Hom_{{\sf Groups}_{/\Aut_\cC(A)}} \bigl( \Omega G_H , \ast \bigr)
\end{eqnarray*}
For $X$ a space, tensoring with it, $X \odot-$, is left adjoint to maps from $X$ to a morphism space.
Specifically, for $U,W \in \cU$ objects in an $\infty$-category that admits colimits indexed by spaces, there is a canonical equivalence between spaces of morphisms:
$\Hom_{\cU}(X \odot U , V) \simeq \Hom_{\Spaces} \bigl( X , \Hom_{\cU}(U,V) \bigr)$.
This first equivalence follows from this adjunction.

The second equivalence follows from the assumption that $H\xra{\rho}G$ is surjective on path-components.

The third equivalence follows from the straightening-unstraightening equivalence: 
\[
(-)_{H}
\colon
\Mod_H(\Spaces) 
\xra{~\simeq~} 
\Spaces_{/\sB H}
~,\qquad
(H \lacts Z)
\mapsto 
Z_H
~,
\]
given by taking $H$-coinvariants.

Composing with the given map $\sB H \to  {\sf BAut}_\cC(A)$ defines a functor:
\[
\Spaces_{/\sB H}
\longrightarrow
\Spaces_{/{\sf BAut}_\cC(A)}
~,\qquad
(Z\to \sB H)
\mapsto 
\bigl(Z \to \sB H \to {\sf BAut}_\cC(A) \bigr)
~.
\]
This functor is a left adjoint, with right-adjoint given by base-change along $\sB H \to  {\sf BAut}_\cC(A)$.
The fourth equivalence then follows from this adjunction, upon observing the base-change square:
\[
\xymatrix{
{\sf Aut}_\cC(A)_H
\ar[rr]
\ar[d]
&&
\ast
\ar[d]
\\
\sB H
\ar[rr]
&&
{\sf BAut}_\cC(A)
.
}
\]

Recall the adjunction
\[
\sB
\colon
{\sf Groups}
~\rightleftarrows~
\Spaces^{\ast/}
\colon
\Omega
~.
\]
In this adjunction, the left adjoint $\sB$ is fully faithful and its image consists of the path-connected pointed spaces.
Because $H\xra{\rho} G$ is surjective on path-components, the space of $H$-coinvariants $G_H$ is path-connected.
The final equivalence then follows from this adjunction.

\end{proof}

We now prove Proposition~\ref{t10}, which we restate as the following.
\begin{theorem}
\label{t19}
There is a canonical $\sO(n-1)\times \sO(1)$-equivariant extension of the action map of~(\ref{e11}):
\[
\xymatrix{
&
\sO(n-1)\times \sO(1)
\times 
\Alg_{{n-1}}^{\rigid}( \Cat_{(\infty,1)} )
\ar[dr]^-{\rm (\ref{e11})'s~action}
\ar[dl]_-{\oplus \times \id }
&
\\
\sO(n)
\times 
\Alg_{{n-1}}^{\rigid}( \Cat_{(\infty,1)} )
\ar@{-->}[rr]^-{\sO(n-1)\times \sO(1)- \rm equivariant}
&
&
\Alg_{{n-1}}^{\rigid}( \Cat_{(\infty,1)} )
~.
}
\]

\end{theorem}

\begin{proof}
Consider the monomorphism $\sO(n-1)\times \sO(1) \xra{\oplus} \sO(n)$ between Lie groups given by block-sum of matrices.
This morphism is surjective on path-components.
The linear action $\sO(n) \lacts \RR^n$ determines a transitive continuous action $\sO(n) \lacts \RR\PP^{n-1}$ with isotropy of the base point $\ast \xra{\lag \RR \subset \RR^n \rag} \RR\PP^{n-1}$ identified as the compact Lie subgroup which is the image of $\oplus$.  
The orbit-stabilizer theorem thusly identifies the $\sO(n-1)\times \sO(1)$-coinvariants of the underlying space $\sO(n)$:
\[
\sO(n)_{\sO(n-1)\times \sO(1)}
\xra{~\simeq~}
\RR\PP^{n-1}
~.
\]
Through this equivalence, the canonical map $\sO(n)_{\sO(n-1)\times \sO(1)} \to \BO(n-1) \times \BO(1)$ is identified as the map $\RR\PP^{n-1} \xra{(L\subset \RR^n)\mapsto (L^{\perp} , L)} \BO(n-1) \times \BO(1)$.  
By Observation~\ref{t30}, this map factors as the composition of the upper horizontal maps in the diagram among pointed spaces:
\begin{equation}
\label{e26}
\xymatrix{
\RR\PP^{n-1}
\ar[rr]
\ar[d]
&&
\RP^{n-2}_+\wedge\BO(1)
\ar[d]_-{(\ref{e34})}
\ar[rr]^-{( {\sf Ref} , \pr )}
&&
\BO(n-1) \times \BO(1)
\ar[d]^-{\Bigl\lag \sO(n-1)\times \sO(1)
\underset{(\ref{e11})} \lacts
\Alg^{\rigid}_{{n-1}}( \Cat_{(\infty,1)} )
\Bigr\rag}
\\
\ast
\ar[rr]
&&
\SS^{n-2}_+ \underset{\sO(1)}\wedge {\sf BU}(1)
\ar[rr]
&&
{\sf BAut}_{\CAT}\Bigl(
\Alg_{{n-1}}^{\rigid}( \Cat_{(\infty,1)} )
\Bigr)
.
}
\end{equation}
This Lemma~\ref{t27} supplies the left commutative square in~(\ref{e26}), while Lemma~\ref{prop.key} supplies  the right commutative square in~(\ref{e26}).
Applying $\Omega$ to the outer square of~(\ref{e26}) is a commutative diagram
\begin{equation}
\label{e17}
\xymatrix{
\Omega \RR\PP^{n-1}
\ar[rr]
\ar[d]
&&
\sO(n-1)\times \sO(1)
\ar[d]^-{(\ref{e11})}
\\
\ast
\ar[rr]
&&
\Aut_{\CAT}\Bigl(
\Alg_{{n-1}}^{\rigid}( \Cat_{(\infty,1)} )
\Bigr)
.
}
\end{equation}
Through Lemma~\ref{t18} applied to the continuous homomorphism $\sO(n-1)\times \sO(1) \xra{\oplus} \sO(n)$, the commutative diagram~(\ref{e17}) is the sought equivariant extension as in the statement of the theorem.

\end{proof}

\begin{cor} 
\label{t23}
The $\sO(n-1)\times\sO(1)$-equivariant functor of Theorem~\ref{t19}
extends as:
\begin{equation}
\label{e30}
\xymatrix{
&
\sO(n-1)\times\sO(1) \times \Alg_{{n-1}}^{\rigid}( \Cat_{(\infty,1)} )
\ar[dl]
\ar[dr]
\ar[dd]
\\
\sO(n) \times \Alg_{{n-1}}^{\rigid}( \Cat_{(\infty,1)} )
\ar[dd]_-{\sO(n-1) \text{\rm -quotient}}
\ar[rr]
&&
\Alg_{{n-1}}^{\rigid}( \Cat_{(\infty,1)} )
\ar[dd]^-{\rm forget}
\\
&
\sO(1) \times \Alg_{{n-1}}^{\rigid}( \Cat_{(\infty,1)} )
\ar[dr]
\ar[dd]
\ar[dl]
\\
\SS^{n-1} \times \Alg_{{n-1}}^{\rigid}( \Cat_{(\infty,1)} )
\ar[dd]_-{\sO(1) \text{\rm - quotient}}
\ar@{-->}[rr]
&&
\Cat_{(\infty,1)}
\ar[dd]^{{\sf ev}_{[p]}}
\\
&
\Alg_{{n-1}}^{\rigid}( \Cat_{(\infty,1)} )
\ar[dl]
\ar[dr]
\\
\RP^{k}\times \Alg_{{n-1}}^{\rigid}( \Cat_{(\infty,1)} )
\ar@{-->}[rr]
&&
\Spaces
}
\end{equation}
in which the left and middle upper vertical functors witness $\sO(n-1)$-quotients, the left and middle lower vertical functors witness $\sO(1)$-quotients, and the middle horizontal functor is $\sO(1)$-equivariant.

\end{cor}

\begin{proof}
Recall the construction of the action $\sO(n-1) \times \sO(1) \underset{(\ref{e11})} \lacts \Alg_{{n-1}}^{\rigid}(\Cat_{(\infty,1)})$ summarized as Observation~\ref{t16}.
In particular, recall that the $\sO(n-1)$-action on $\Alg_{{n-1}}(\Cat_{(\infty,1)})$ corepresented by an $\sO(n-1)$-action on the $\infty$-operad $\cE_{n-1}$.
It follows that the right upper vertical functor is $\sO(n-1)$-invariant.
Also, recall that the $\sO(1)$-action on $\Cat_{(\infty,1)}$ is corepresented by an involution on $\bDelta$ given by reversing linear orders.
Note the unique equivalence $[p]^{\op} \cong [p]$ in $\bDelta$ for each $[p]\in \bDelta$.
These unique equivalences reveal that his involution on $\bDelta$ restricts as the identity involution on $\Obj(\bDelta)$.
It follows that, for each $[p]\in \bDelta$, the evaluation functor $\Cat_{(\infty,1)} \xra{\ev_{[p]}} \Spaces$ is $\sO(1)$-invariant.

Now, given these invariances of the vertical functors in~(\ref{e30}), 
the equivariant fillers in~(\ref{e30}) immediately follow from the $\sO(n-1)\times \sO(1)$-equivariance of the filler in Theorem~\ref{t19}, using the standard identities:
\[
\sO(n)_{\sO(n-1) } \simeq \SS^{n-1}
\qquad
\text{ and }
\qquad
\sO(n)_{\sO(n-1) \times \sO(1)}
~\simeq~ 
\SS^{n-1}_{\sO(1)} \simeq \RR\PP^{n-1}
~.
\]

\end{proof}

\subsection{Case $d=1$ and $g=\infty$}
\label{sec.main.proof}

We now state and prove our main theorem (Theorem~\ref{t13} above) in the cases in which $d=1$ and $g=\infty$, which we state more precisely here.
\begin{theorem}
\label{t24}
Let $p\geq 0$ be a non-negative integer.
There are canonical lifts as in the diagram among $\infty$-categories:
\begin{equation}
\label{e29}
\xymatrix{
\Alg_{{n-1}}^{\rigid}( \Cat_{(\infty,1)} )
\ar@{-->}[rrr]
\ar@{-->}[drr]
\ar@{-->}[ddr]
\ar@(d,d)[dddr]_-{\rm inclusion}
&
&
&
\Mod_{\Omega \RR\PP^{n-1}} ( \Spaces )
\ar[d]^-{\rm forget}
\\
&
&
\Mod_{\Omega \SS^{n-1}} ( \Cat_{(\infty,1)} )
\ar[d]_-{\rm forget}
\ar[r]
&
\Mod_{\Omega \SS^{n-1}} ( \Spaces )
\ar[d]^-{\rm forget}
\\
&
\Mod_{\Omega \sO(n)} \Bigl( \Alg_{{n-1}}^{}( \Cat_{(\infty,1)} ) \Bigr)
\ar[r]
\ar[d]^-{\rm forget}
&
\Mod_{\Omega \sO(n)} ( \Cat_{(\infty,1)} )
\ar[r]
\ar[d]^-{\rm forget}
&
\Mod_{\Omega \sO(n)} ( \Spaces )
\ar[d]^-{\rm forget}
\\
&
\Alg_{{n-1}}^{}( \Cat_{(\infty,1)} )
\ar[r]^-{\rm forget}
&
\Cat_{(\infty,1)}
\ar[r]^-{\ev_{[p]}}
&
\Spaces
}
\end{equation}
In particular, for $\cR$ an $(n-1)$-monoidal $(\infty,1)$-category, provided $\cR$ is rigid then there exists a canonical fillers in the diagram among groups:
\[
\xymatrix{
\Omega\sO(n)
\ar@{-->}[rr]
\ar[d]_-{\Omega \bigl( \sO(n-1) \text{\rm -quotient} \bigr)}
&&
\Aut_{\Alg_{{n-1}}( \Cat_{(\infty,1)} )}(\cR)
\ar[d]
\\
\Omega \SS^{n-1}
\ar@{-->}[rr]
\ar[d]_-{\Omega \bigl( \sO(1) \text{\rm -quotient} \bigr)}
&&
\Aut_{\Cat_{(\infty,1)}}(\cR)
\ar[d]
\\
\Omega\RP^{n-1}
\ar@{-->}[rr]
&&
\Aut_{\Spaces}(\cR[p])
.
}
\]
\end{theorem}

\begin{proof}
Let $\cA\to \cB$ be a functor between $\infty$-categories; 
let $x\in K$ be a pointed space.
Observe the canonical maps between spaces of morphisms:
\begin{eqnarray}
\label{e31}
\Hom_{(\Cat_{\infty})^{\cA/}} ( K \times \cA , \cB )
&
~\simeq~
&
\Hom_{(\Cat_{\infty})_{/\cB}} \bigl( \cA , \Fun(K , \cB) \bigr)
\\
\nonumber
&
\longrightarrow
&
\Hom_{(\Cat_{\infty})_{/\cB}} \bigl( \cA , \Fun(\sB \Omega K , \cB) \bigr)
\\
\nonumber
&
~=~
&
\Hom_{(\Cat_{\infty})_{/\cB}} \bigl( \cA , \Mod_{\Omega K}(\cB) \bigr)
~,
\end{eqnarray}
where the first equivalence follows from a Tensor-Hom adjunction, and the second map is induced by the connected cover $\sB \Omega K \to K$, and the third identity is the definition of $\Mod_{\Omega K}(\cB)$.
\begin{itemize}
\item
Consider the case of the functor $\Alg_{{n-1}}^{\rigid}( \Cat_{(\infty,1)} ) \xra{\id} \Alg_{{n-1}}^{\rigid}( \Cat_{(\infty,1)} )$ and the pointed space $\uno \in \sO(n)$.
In this case, the upper horizontal functor in~(\ref{e30}) is carried through the map~(\ref{e31}) to a functor
\[
\Alg_{{n-1}}^{\rigid}( \Cat_{(\infty,1)} )
\longrightarrow
\Mod_{\Omega \sO(n)} \bigl( \Alg_{{n-1}}^{\rigid}( \Cat_{(\infty,1)} ) \bigr)
\]
over $\Alg_{{n-1}}^{\rigid}( \Cat_{(\infty,1)} ) \xra{\id} \Alg_{{n-1}}^{\rigid}( \Cat_{(\infty,1)} )$.

\item
Consider the case of the functor $\Alg_{{n-1}}^{\rigid}( \Cat_{(\infty,1)} ) \xra{\rm forget} \Cat_{(\infty,1)}$ and the pointed space $\ast \in \SS^{n-1}$.
In this case, the middle horizontal functor in~(\ref{e30}) is carried through the map~(\ref{e31}) to a functor
\[
\Alg_{{n-1}}^{\rigid}( \Cat_{(\infty,1)} )
\longrightarrow
\Mod_{\Omega \SS^{n-1} } ( \Cat_{(\infty,1)} )
\]
over $\Alg_{{n-1}}^{\rigid}( \Cat_{(\infty,1)} ) \xra{\rm forget} \Cat_{(\infty,1)}$.

\item
Consider the case of the functor $\Alg_{{n-1}}^{\rigid}( \Cat_{(\infty,1)} ) \xra{\ev_{[p]} \circ {\rm forget}} \Spaces$ and the pointed space $\ast \in \RR\PP^{n-1}$.
In this case, the lower horizontal functor in~(\ref{e30}) is carried through the map~(\ref{e31}) to a functor
\[
\Alg_{{n-1}}^{\rigid}( \Cat_{(\infty,1)} )
\longrightarrow
\Mod_{\Omega \RR\PP^{n-1} } ( \Spaces )
\]
over $\Alg_{{n-1}}^{\rigid}( \Cat_{(\infty,1)} ) \xra{\ev_{[p]} \circ {\rm forget}} \Spaces$.

\end{itemize}
This supplies the dashed functors in the diagram~(\ref{e29}), each for which they fill the solid diagram of~(\ref{e29}).
Commutativity of the entire diagram~(\ref{e29}) follows immediately from commutativity of the right-angled squares in diagram~(\ref{e30}) of Corollary~\ref{t23}.

\end{proof}

\subsection{Case of 1-categories in $\cV$}
\label{sec.internal}
Here, we establish a version of our main theorem in which $\Cat_{(\infty,1)}$ is replaces by $\Cat_1[\cV]$.

\begin{convention}
In this section, we fix an $\infty$-category $\cV$ with finite products.

\end{convention}

\begin{example}
A case of particular interest is $\cV = \Cat_{(\infty,d-1)}$ for a choice of non-negative integer $d> 0$, with the $d=1$ instance being $\cV = \Cat_{(\infty,0)} = \Spaces$.

\end{example}

\begin{definition}
\label{d2}
A \bit{1-category in $\cV$} is a simplicial object $\cC \colon \bDelta^{\op} \to \cV$ satisfying the following conditions.
\begin{itemize}
\item
The Segal condition: 
for each $[p]\in \bDelta$, the functor $\cC$ carries the opposite of the diagram 
(\ref{e51})
in $\bDelta$
to a limit diagram in $\cV$.

\item
The univalent-completeness: the functor $\cC$ carries the opposite of the diagram
(\ref{e52})
in $\bDelta$
to a limit diagram in $\cV$.
\end{itemize}
For $\cC$ a 1-category in $\cV$, its \bit{objects} and its \bit{morphisms} are the respective values:
\[
\Obj(\cC)
~:=~ 
\cC([0])
~ \in \cV
\qquad
\text{ and }
\qquad
\Mor(\cC)
~:=~ 
\cC([1])
~\in \cV
~.
\]
The $\infty$-category of \bit{1-categories in $\cV$} is the full $\infty$-subcategory 
\[
\Cat_1[\cV]
~\subset~
\Fun(\bDelta^{\op} , \cV)
\]
consisting of the 1-categories in $\cV$.

\end{definition}

\begin{example}
\label{r6}
The definition of a 1-category in $\Spaces$ as a complete Segal space (\cite{rezk}) lends an identity between $\infty$-categories: 
\[
\Cat_{(\infty,1)} 
~\simeq~ 
\Cat_1[\Spaces]
~.
\]

\end{example}

For $d\geq 0$, recall from~\cite{rezk.n} the definition of the $(\infty,1)$-category $\Cat_{(\infty,d)}$ of $(\infty,d)$-categories.
\begin{definition}
\label{d7}
Let $d \geq 0$ and $d\leq g \leq \infty$.
The $(\infty,1)$-category of $(g,d)$-categories is the full $(\infty,1)$-subcategory
\[
\Cat_{(g,d)}
~\subset~
\Cat_{(\infty,d)}
\]
consisting of those $(\infty,d)$-categories $\cC$ in which, for each pair of objects $x,y\in \Obj(\cC)$ the following property holds.
\begin{itemize}

\item
If $d=0$, the space $\Hom_{\cC}(x,y)$ is a $(g-1)$-type;

\item
If $d\geq 1$, the $(\infty,d-1)$-category $\un{\Hom}_{\cC}(x,y)$ is a $(g-1,d-1)$-category.

\end{itemize}

\end{definition}

\begin{observation}
\label{t40}
Let $d \geq 0$ and $d\leq g \leq \infty$.
The fully faithful inclusion $\Cat_{(g,d)} \hookrightarrow \Cat_{(\infty,d)}$ is a right adjoint in a localization, and in particular detects and preserves limits.  
Furthermore, the left adjoint in this localization preserves finite products.

\end{observation}

\begin{example}
\label{r9}
Let $d\geq 1$ and $d\leq g \leq \infty$.
In the case that $\cV = \Cat_{(g,d-1)}$, the $\infty$-category of $(g,d)$-categories can be presented as the full $\infty$-subcategory
\[
\Cat_{(g,d)}
~\subset~
\Cat_1[\Cat_{(g,d-1)}]
\]
consisting of those 1-categories in $\Cat_{(g,d-1)}$, denoted $\bDelta^{\op} \xra{\cC} \Cat_{(g,d-1)}$, for which 
\begin{itemize}
\item
the value $\cC([0]) \in \Spaces_{\leq g} = \Cat_{(g,0)} \subset \Cat_{(g,d-1)}$ is a $g$-type;

\item
each fiber of the map $\cC([1]) \to \cC(\{0\})\times \cC(\{1\})$, which is a priori a $(g,d-1)$-category, is in fact a $(g-1,d-1)$-category.\footnote{In other words, for each $U,V\in \cC([0])$, the a priori $(g,d-1)$-category $\un{\Hom}_{\cC}(U,V)$ is in fact a $(g-1,d-1)$-category.}\footnote{Note that  this second condition is vacuous in the case that $g=\infty$.}

\end{itemize}  
Note that these two conditions are preserved by finite products.
Therefore the above is a full $\infty$-subcategory.

\end{example}

\begin{example}
\label{r5}
For $\cK$ an $\infty$-category, there is a canonical identification between $\infty$-categories,
\[
\Cat_1\bigl[\PShv(\cK)\bigr] 
~\simeq~ 
\Fun\bigl(\cK^{\op} , \Cat_1[\Spaces]\bigr) 
\underset{\rm Eg~\ref{r6}}{~\simeq~}
\Fun(\cK^{\op} , \Cat_{(\infty,1)})
~,
\]
in which the first identification is a restriction of $\Fun\bigl(\bDelta^{\op},\Fun(\cK^{\op} , \Spaces) \bigr) \simeq \Fun\bigl(\cK^{\op},\Fun(\bDelta^{\op} , \Spaces)$.

\end{example}

\begin{observation}
\label{t17}
Using that $\cV$ admits finite products, the $\infty$-category $\Cat_1[\cV]$ admits finite products and the fully faithful functor
\[
\Cat_1[\cV]
~\hookrightarrow~
\Fun(\bDelta^{\op} , \cV)
\]
preserves finite products.
In particular, the $\infty$-category $\Cat_1[\cV]$ admits a Cartesian symmetric monoidal structure.  

\end{observation}

Observation~\ref{t17} validates the following.
\begin{convention}
\label{d3}
We regard $\Cat_1[\cV]$ as a symmetric monoidal $\infty$-category via its Cartesian symmetric monoidal structure.

\end{convention}

\begin{definition}
\label{d5}
The $\infty$-category of \bit{$(n-1)$-monoidal $1$-categories in $\cV$} is
\[
\Alg_{{n-1}}\bigl( \Cat_1[\cV] \bigr)
~.
\]

\end{definition}

\begin{observation}
\label{t26}
The assignment $\cW \mapsto \Cat_1[\cW]$ assembles as a functor
\[
\CAT^{\sf f.lim}
\longrightarrow
\CAT
~,\qquad
\cW
\mapsto 
\Cat_1[\cW]
~,
\]
to the $\infty$-category of $\infty$-categories and functors between them
from the $\infty$-category of $\infty$-categories and functors between them that preserve finite limits that exist.
For instance, a finite-limit-preserving functor $\cU \xra{F} \cW$ between $\infty$-categories determines the symmetric monoidal functor
\[
\Cat_1[\cU]
\xra{~\Cat_1[F]~}
\Cat_1[\cW]
~,\qquad
(\bDelta^{\op} \xra{\cC} \cU)
\mapsto
(\bDelta^{\op} \xra{\cC} \cU \xra{F} \cW)
~.
\]

\end{observation}

\begin{observation}
\label{r4}
The Yoneda functor $\cV \xra{j} \PShv(\cV)$ preserves and detects finite limits and is fully faithful.
Through Observation~\ref{t26} and Example~\ref{r5}, there results a fully faithful symmetric monoidal functor:
\[
\Cat_1[\cV]
\xra{~\Cat_1[j]~}
\Fun( \cV^{\op} , \Cat_{(\infty,1)})
~.
\]

\end{observation}

\begin{observation}
\label{t31}
There is a canonical identification between $\infty$-categories:
\[
\Alg_{{n-1}} \bigl( \Fun(\cV^{\op} , \Cat_{(\infty,1)}) \bigr)
~\simeq~
\Fun\bigl( \cV^{\op} , \Alg_{{n-1}} (\Cat_{(\infty,1)}) \bigr)
~.
\]

\end{observation}

\begin{definition}
\label{d6}
The $\infty$-category of \bit{rigid $(n-1)$-monoidal $1$-categories in $\cV$} is the full $\infty$-subcategory
\[
\Alg_{{n-1}}^{\rigid} \bigl( \Cat_1[\cV] \bigr)
~\subset~
\Alg_{{n-1}}\bigl( \Cat_1[\cV] \bigr)
\]
consisting of those $(n-1)$-monoidal 1-categories $\cR$ in $\cV$ for which, for each $V\in \cV$, the $(n-1)$-monoidal $(\infty,1)$-category $\Hom_{\cV}(V,\cR)$ is rigid.

\end{definition}

\begin{observation}
\label{t34}
There is a defining pullback diagram among $\infty$-categories in which arrow is fully faithful:
\begin{equation}
\label{e49}
\xymatrix{
\Alg_{{n-1}}^{\rigid} \bigl( \Cat_1[\cV] \bigr)
\ar[rr]
\ar[d]
&&
\Alg_{{n-1}}\bigl( \Cat_1[\cV] \bigr)
\ar[d]^-{\Alg_{{n-1}}\bigl( \Cat_1[j] \bigr)}
\\
\Fun\bigl( \cV^{\op} , \Alg_{{n-1}}^{\rigid} (\Cat_{(\infty,1)}) \bigr)
\ar[rr]
&&
\Fun\bigl( \cV^{\op} , \Alg_{{n-1}}(\Cat_{(\infty,1)}) \bigr)
.
}
\end{equation}

\end{observation}

\subsection{Proof of main theorem}
\label{sec.final}

Here we prove the general case of the main theorem (Theorem~\ref{t13}), phrased as Theorem~\ref{t35} below.

\begin{observation}
\label{t36}
The pointwise action $\sO(n-1)\times \sO(1) \underset{(\ref{e11})} \lacts \Alg_{n-1}(\Cat_{(\infty,1)})$ defines an action $\sO(n-1)\times \sO(1) \underset{(\ref{e11})} \lacts \Fun\bigl( \cV^{\op} , \Alg_{n-1}(\Cat_{(\infty,1)}) \bigr)$.
This action restricts
as an action by $\sO(n-1)\times \sO(1)$ on each of the full $\infty$-subcategories in~(\ref{e49}).  
In the case that $\cV = \Cat_{(g,d-1)}$, this action further restricts through Example~\ref{r9} as an action by $\sO(n-1)\times \sO(1)$ on the $\infty$-categories $\Alg_{n-1}(\Cat_{(g,d)})$ and $\Alg_{n-1}^{\rigid}(\Cat_{(g,d)})$.  

\end{observation}

Using that the morphism $\sO(n-1)\times \sO(1) \xra{\oplus} \sO(n)$ is surjective on path-components,  Observation~\ref{t36} lends the following.
%
\begin{theorem}
\label{t35}
Let $d \geq 1$ and $d\leq g \leq \infty$.
Theorem~\ref{t19} and Theorem~\ref{t24} are true for the symmetric monoidal $\infty$-category $\Cat_{(\infty,1)}$ replaced by the symmetric monoidal $\infty$-category $\Cat_1[\cV]$, as well as by the symmetric monoidal $\infty$-category $\Cat_{(g,d)}$ of $(g,d)$-categories.

\end{theorem}

Now, each of the a priori $(\infty,1)$-categories $\Alg^{\rigid}_{{n-1}}(\Cat_{(g,d)})$ and $\Cat_{(g,d)}$ and $\Spaces_{\leq g}$ is in fact a $(g+1,1)$-category.
Finally, the specific statement of Theorem~\ref{t13} follows from Theorem~\ref{t35} and the following.  
\begin{observation}
\label{t32}
Let $d \geq 1$ and $d\leq g \leq \infty$.
Let $X\in \cX$ be an object in a $(g+1,d)$-category.
By Observation~\ref{t40},
the fully faithful inclusion 
$
\Spaces_{\leq g}
\hookrightarrow
\Spaces
$
is a right adjoint in a localization, 
\[
\Spaces
~\rightleftarrows~
\Spaces_{\leq g}
~,
\]
whose left adjoint preserves products.  
Using this, each morphism from a continuous group uniquely factors:
\[
\xymatrix{
G 
\ar[rr]
\ar[dr]
&&
\Aut_{\cX}(X)
\\
&
\pi_{\leq g} G
\ar@{-->}[ur]_-{\exists !}
&
.
}
\]
\end{observation}


\begin{thebibliography}{99}



\bibitem[AF]{n=2} Ayala, David; Francis, John. Symmetries of rigid 2-categories. In preparation.





\bibitem[BaDo]{baezdolan} Baez, John; Dolan, James.
Higher-dimensional algebra and topological quantum field theory.
J. Math. Phys. 36 (1995), no. 11, 6073--6105. 


\bibitem[BV]{bv} Boardman, J. Michael; Vogt, Rainer. Homotopy invariant algebraic structures on topological spaces. Lecture Notes in Mathematics, Vol. 347. Springer-Verlag, Berlin-New York, 1973. x+257 pp.






\bibitem[Du]{dunn} Dunn, Gerald. Tensor products of operads and iterated loop spaces. J. Pure Appl. Algebra 50, 1988, no. 3, 237--258.


\bibitem[Ko]{kontsevich}. Kontsevich, Maxim. Operads and motives in deformation quantization, Lett. Math. Phys. 48 (1999), no. 1, 35–72.


\bibitem[Lu1]{HTT} Lurie, Jacob. Higher topos theory. Annals of Mathematics Studies, 170. Princeton University Press, Princeton, NJ, 2009. xviii+925 pp.

\bibitem[Lu2]{HA} Lurie, Jacob. Higher algebra. Preprint, 2017.

\bibitem[Lu3]{cobordism} Lurie, Jacob. On the classification of topological field theories. Current developments in mathematics, 2008, 129--280, Int. Press, Somerville, MA, 2009.


\bibitem[Ma]{May} May, J. Peter. The geometry of iterated loop spaces. Lectures Notes in Mathematics, Vol. 271. Springer-Verlag, Berlin-New York, 1972. viii+175 pp.



\bibitem[Re1]{rezk}  Rezk, Charles. A model for the homotopy theory of homotopy theory. Trans. Amer. Math. Soc. 353 (2001), no. 3, 973--1007.

\bibitem[Re2]{rezk.n} Rezk, Charles. A Cartesian presentation of weak $n$-categories. Geom. Topol. 14 (2010), no. 1, 521--571.


\bibitem[Ta]{tamarkin} Tamarkin, Dmitry. Formality of chain operad of little discs, Lett. Math. Phys. 66 (1-2):65–72, 2003.


\end{thebibliography}
\end{document}